\documentclass{amsart}

\usepackage{amsmath, amscd, graphicx, amsbsy, amssymb}
\usepackage{faktor}
\usepackage{tikz}
\usepackage{hyperref}

\newcommand{\nset}[1]{\underline{#1}}
\newcommand{\cycles}[1]{\left(#1\right)}

\newcommand{\intersect}[1]{\bigl< #1 \bigr>}

\newcommand{\FiniteCount}[1]{\left| #1 \right|}
 
\newcommand{\abs}[1]{| #1 |}

\newcommand{\Orbit}{\mathcal{O}}
\newcommand{\TorusCover}{\mathcal{U}}
\newcommand{\symplecticQuotient}{//}

\newcommand{\Moduli}{\mathcal{M}}
\newcommand{\CompactModuli}{\overline{\mathcal{M}}}

\newcommand{\Graphs}{\mathcal{G}}
\DeclareMathOperator{\Met}{Met}
\DeclareMathOperator{\Sym}{Sym}
\DeclareMathOperator{\RG}{RG}
\newcommand{\CompactRG}{\overline{\RG}}

\newcommand{\ZZ}{\mathbb{Z}}
\newcommand{\Reals}{\mathbb{R}}
\newcommand{\RR}{\mathbb{R}}

\newcommand{\CC}{\mathbb{C}}
\newcommand{\TT}{\mathbb{T}}
\newcommand{\PP}{\mathbb{P}}
\newcommand{\Torus}{\TT}

\newcommand{\quotient}[2]{\faktor{#1}{#2}}

\DeclareMathOperator{\Vol}{Vol}

\DeclareMathOperator{\Aut}{Aut}
\DeclareMathOperator*{\Res}{Res}
\newtheorem{theorem}{Theorem}[section]

\newtheorem{lemma}[theorem]{Lemma}
\newtheorem{definition}[theorem]{Definition}

\newcommand{\disjoint}{\sqcup}

\numberwithin{equation}{section}

\allowdisplaybreaks

\begin{document}
\author{Julia Bennett}
\address{Bard College}
\email{juliacbennett@gmail.com}

\author{David Cochran}
\address{Virginia Commonwealth University}
\email{cochrandv@vcu.edu}

\author{Brad Safnuk}
\address{Central Michigan University}
\email{brad.safnuk@cmich.edu}

\author{Kaitlin Woskoff}
\address{Hartwick College}
\email{woskoffk@hartwick.edu}

\thanks{B.S. would like to thank Motohico Mulase and Yongbin Ruan for helpful comments and suggestions.
  Part of the research for this work was conducted as part of Central
    Michigan University's REU program in the summer of 2009. J.B.,
	D.C.
	  and K.W. were partially supported by NSF-REU grant \#DMS
	  08-51321.
	  }

\title[Topological recursion for symplectic volumes of {$\Moduli_{g,n}$}]{Topological recursion for symplectic volumes of moduli spaces of curves}
\begin{abstract}
We construct locally defined symplectic torus actions on ribbon graph
complexes. Symplectic reduction techniques allow for a recursive
formula for the symplectic volumes of these spaces. Taking the Laplace
transform results in the Eynard-Orantin recursion formulas for the
Airy curve $x = \frac{1}{2}y^2$.
\end{abstract}

\date{\today}
\maketitle
\tableofcontents

\section{Introduction}
Since Kontsevich's proof \cite{MR1171758} of the Witten conjecture \cite{MR1144529},
there has been a flurry of activity centered around the tautological
ring of the moduli space of curves, and expanded more generally to
Gromov-Witten invariants. However, many of the fundamental tools
developed by Kontsevich have remained comparatively ignored. 

In this
paper we focus on the combinatorially defined 2-form
$\Omega_{\vec{L}}$ used by
Kontsevich to represent the scaled sum of $\psi$-classes
\begin{equation*}
  [\Omega_{\vec{L}}] = \frac{1}{2}(L_1^2 \psi_1 + \cdots L_n^2 \psi_n).
\end{equation*}
In particular, this form leads to a family of symplectic structures on
the moduli space of curves, with the associated volumes encoding all
possible $\psi$-class intersection numbers. Although the
non-degeneracy of $\Omega$ appeared in Kontsevich's original work, the
symplectic nature of $\Omega$ was not taken advantage of in any
particular way.

We develop a recursive formula (an example of \emph{topological
recursion}, as explained below) for calculating the symplectic volume
of the moduli space of curves. In particular, if $\Vol_{g,n}(L_1,
\ldots, L_n)$ represents the symplectic volume of
$\CompactModuli_{g,n}$, calculated with repspect to the symplectic
form $\Omega_{\vec{L}}$, then we have
\begin{theorem}
  The symplectic volumes of moduli spaces of curves obey the recursion
  relation
\begin{equation}
  \begin{split}
L_1 & \Vol_{g,n}(L_1,  \ldots, L_n) \\
 &  = \sum_{j=2}^n
	\int_{\abs{L_1-L_j}}^{L_1 + L_j}dx\, \frac{x}{2} (L_1 + L_j - x ) 
	\Vol_{g,n-1}(x, L_2, \ldots, \hat{L}_j, \ldots, L_n) \\
 & \quad  + \sum_{j=2}^n \int_{0}^{\abs{L_1-L_j}} dx\, x f(x, L_1, L_j)
 \Vol_{g,n-1}(x, L_2, \ldots, \hat{L}_j, \ldots, L_n) \\
 & \quad + \iint_{0 \leq x+y \leq L_1} dxdy\,\frac{xy}{2}(L_1 - x-y)
	\Vol_{g-1,n+1}(x,y, L_2, \ldots, L_n) \\
 &  + \sum_{\substack{g_1 + g_2 = g \\ \mathcal{I} \disjoint \mathcal{J} = \nset{n}\setminus 1}}
	\iint_{0 \leq x+y \leq L_1} dxdy\,\frac{xy}{2}(L_1 - x-y)
	\Vol_{g_1,n_1}(x, L_{\mathcal{I}})\Vol_{g_2,n_2}(y, L_{\mathcal{J}}),
  \end{split}
	\label{eq:IntroRecursionFormula}
\end{equation}
  subject to the initial conditions
  \begin{align}
	\Vol_{0,3}(L_1, L_2, L_3) &= 1 \label{eq:IntroVol-0-3}\\
	\Vol_{1,1}(L) &= \frac{1}{48}L^2 \label{eq:IntroVol-1-1},
  \end{align}
  and $\Vol_{g,n}(L_1, \ldots, L_n) = 0$ if $2g -2 + n <= 0$.
\end{theorem}

The key technique used in the proof involves constructing Hamiltonian
torus actions which act locally on moduli space (cf \cite{MR2503971} for a
related, but different toric symmetry on moduli of curves).

We show that the above recursion has as a simple corollary the DVV
formula \cite{MR1083914} for $\psi$-class intersections:
\begin{equation}
  \begin{split}
	\intersect{\tau_{d_1} \cdots \tau_{d_n}}_{g} &=
	\sum_{j=2}^{n} \frac{(2d_1 + 2d_j - 1)!!}{(2d_1 + 1)!! (2d_j -
	1)!!} \intersect{\tau_{d_1 + d_j - 1} \tau_{d_{\nset{n} \setminus
	\{1, j\} }}}_{g} \\
	& \quad + \frac{1}{2} \sum_{a+b = d_1 - 2}
	\frac{(2a+1)!!(2b+1)!!}{(2d_1+1)!!}
	\Biggl[ \intersect{\tau_a\tau_b \tau_{d_{\nset{n}\setminus
	1}}}_{g-1}  \\
	&  \quad\quad\quad + \sum_{\substack{g_1 + g_2 = g \\ \mathcal{I} \sqcup
	\mathcal{J} = \nset{n} \setminus 1}}^{\text{stable}} 
	\intersect{\tau_a\tau_{d_{\mathcal{I}}}}_{g_1}
	\intersect{\tau_b \tau_{d_{\mathcal{J}}}}_{g_2}\Biggr].
  \end{split}
  \label{eq:IntroDVV}
\end{equation}
thus providing yet another proof of the Witten-Kontsevich theorem.

In addition, by defining 
\begin{equation*}
  W_{g,n}(z_1, \ldots, z_n) = \int_{\RR_+^n} e^{-\sum z_i L_i} 
   \Vol_{g,n}(L_1, \ldots, L_n)\prod L_i\,dL_i,
\end{equation*}
and taking the Laplace transform of recursion relation
\eqref{eq:IntroRecursionFormula} we arrive at the equivalent recursion
formula
\begin{equation}
  \begin{split}
  W_{g,n}(z_1, \ldots, z_n) &= \sum_{j=2}^{n} -\frac{\partial}{\partial z_j}
  \biggl[ \frac{z_j}{(z_1 z_j)^2(z_1^2 - z_j^2)} 
  \Bigl( z_1^2 W_{g, n-1}(z_2, \ldots, z_n) \\
  & \quad\quad\hspace{15mm}- z_j^2 W_{g, n-1}(z_1, \ldots, \hat{z}_j, \ldots, z_n)
  \Bigr)\biggr] \\
  & \quad+ \frac{1}{2z_1^2} W_{g-1, n+1}(z_1, z_1, \ldots, z_n) \\
  & \quad+ \frac{1}{2z_1^2} \sum_{\substack{g_1 + g_2 = g \\ \mathcal{I}
  \sqcup \mathcal{J} = \nset{n} \setminus 1}} 
  W_{g_1, n_1}(z_1, z_{\mathcal{I}})W_{g_2, n_2}(z_1,
  z_{\mathcal{J}}),
\end{split}
  \label{eq:IntroEO_Recursion}
\end{equation}
We prove that \eqref{eq:IntroEO_Recursion} is an example of the
Eynard-Orantin recursion formula \cite{MR2346575} for the spectral curve
$x = \frac{1}{2}y^2$.

We should emphasize that, apart from recursion equation
\eqref{eq:IntroRecursionFormula}, none of the results of the paper are
new. For example, there are by now many proofs of the
Witten-Kontsevich theorem \cite{MR1171758, MR2483941, 2005math8384K,
MR2257394, MR2328716, MR2503971, mulase-2009}, several of which use techniques
similar to what was done in the present work. In addition, it has been
shown by Eynard and Orantin \cite{MR2519749} that the Airy curve encodes the
$\psi$-class intersection numbers.

Our aim then is not to produce new results in a well-mined
field, but rather present a novel point of view which has wider
applicability and ramifications. For example, our work makes it
geometrically clear \emph{why} it is that the Airy curve encodes intersection
numbers - a point of view lacking in the literature. In addition, the
techniques developed have a much wider applicability. For example,
similar ideas can be used to motivate a generalization of
Eynard-Orantin invariants \cite{Safnuk-2010} which captures both the
generalized Kontsevich matrix model (and in the process intersection
numbers of $\psi$-classes over Witten cycles) and intersection theory
for $r$-spin curves. As well, although the Airy curve is the simplest
non-trivial example of the Eynard-Orantin invariants, it is also
universal, in the sense that locally all spectral curves look like the
Airy curve. Having a good understanding of the local structure of
Eynard-Orantin invariants allows one to extrapolate to arbitrary
spectral curves by a perturbation type argument
\cite{SafnukMulase-2010}. As well, it should be pointed out that the
recursion formula proven here played an important role in deriving a
new proof \cite{2010arXiv1009.2055C} of Kontsevich's integration constant $\rho =
2^{5g-5+2n}$, first appearing in \cite{MR1171758}, relating the symplectic volume of the ribbon graph complex to
the Euclidean push-foward measure.

This paper is organized as follows. In Section~\ref{sect:Background} we survey
the definitions and constructions needed in the paper. We define the
ribbon graph complex, and the symplectic 2-form $\Omega$ originally
constructed by Kontsevich. We discuss the relationship to tautological
classes on the moduli space of stable curves, and also consider the
Eynard-Orantin invariants, focusing on the relevant case of when the
spectral curve is $\PP^1$. Finally, we survey the tools from
symplectic geometry which are necessary in the sequel.
In Section~\ref{sect:LocalStructure}, we construct the local torus symmetries on the
ribbon graph complex and show that the associated symplectic quotients
are also ribbon graph complexes. In Section~\ref{sect:RecursionFormula}, 
we use the local picture to derive
recursion equation \eqref{eq:IntroRecursionFormula}, and provide full consideration of the
base case volumes \eqref{eq:IntroVol-0-3} and \eqref{eq:IntroVol-1-1}.
In Section~\ref{sect:Virasoro}
we prove that our recursion relation is equivalent to the DVV equation
(Virasoro constraint) for $\psi$-class intersections on
$\CompactModuli_{g,n}$, while in Section~\ref{sect:EO_Recursion} we prove that it
is equivalent to the Eynard-Orantin recursion for the spectral curve
$x=\frac{1}{2}y^2$. 
\section{Background}
\label{sect:Background}
\subsection{Ribbon graph complexes}

A ribbon graph is a graph with a cyclic ordering assigned to the
half-edges incident on each vertex. The cyclic ordering allows the
edges of the graph to be fattened in a canonical way into ribbons,
with the resulting surface having an orientation which induces
the cyclic ordering at each vertex. Some examples, along with the
associated surfaces, are presented in
Figure~\ref{fig:RibbonGraphExamples}, where the cyclic ordering is
implied from the standard counter-clockwise orientation of the plane.
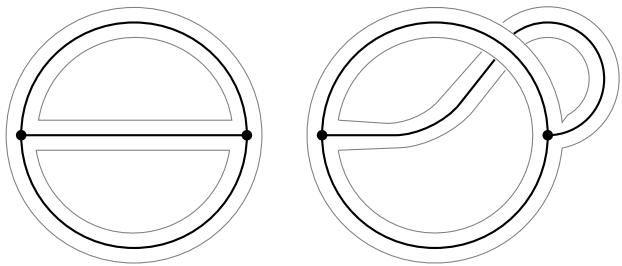
\begin{figure}
  \begin{tikzpicture}
	\draw[thick] (-2, 0) circle (1.5cm);
	\draw[thick] (-3.5, 0) -- (-0.5, 0);
	\draw[thick] (2, 0) circle (1.5cm);
	\fill (-3.5, 0) circle (2pt);
	\fill (-0.5, 0) circle (2pt);
	\draw[gray] (-2,0) circle (1.7cm);
	\draw[gray] (-0.7,0.2) arc (8.8:171.2:1.3cm) -- cycle;
	\draw[gray] (-3.3,-0.2) arc (188.8:351.2:1.3cm) -- cycle;
	\draw[thick, rounded corners=15pt] (3.5, 0.75) +(-90:0.75) arc (-90:127:0.75)
	  (3.5, 0.75) +(160:0.75) -- (2, 0) -- (0.5, 0);
	\fill (0.5, 0) circle (2pt);
	\fill (3.5, 0) circle (2pt);
	\draw[gray] (3.5, 0.75) +(122:0.55) arc (122:-61.2:0.55) -- (3.6913, 0.1718) arc (5.8:354.2:1.7)  
	arc (-78.3:129:0.95);
	\draw[gray] (2, 0) +(-171.2:1.3) arc (-171.2:171.2:1.3);
	\draw[gray, rounded corners=12pt] (3.5, 0.75) ++(162:0.75)
	+(-0.164, 0.13) -- (1.7586, 0.1314) -- (0.72, 0.2);
	\draw[gray, rounded corners=17pt] (3.5, 0.75) ++(160:0.75)
	+(0.1414, -0.1414) -- (2.1414, -0.1414) -- (0.72, -0.2);
  \end{tikzpicture}
\caption{Ribbons graphs of type $(0,3)$ and $(1,1)$.}
\label{fig:RibbonGraphExamples}
\end{figure}

A more precise way of defining ribbon graphs, which better elucidates
their automorphisms, comes from using permutation data. Let $\gamma \in
S_k$ be a permutation of the set $\nset{k} = \{1, 2, \ldots, k\}$. Then
the notation $\cycles{\gamma}$ represents the set of disjoint orbits
(cycles) or $\gamma$, and $\FiniteCount{\cycles{\gamma}}$ denotes the
number of orbits. For example, if $\gamma = (134)(2)(56)$ then
$(\gamma) = \left\{ (134), (2), (56) \right\}$, while
$\FiniteCount{\cycles{\gamma}} = 3$.

\begin{definition}
A \emph{ribbon graph} is a collection $(\gamma_0, \gamma_1, \gamma_2, b)$
such that
\begin{enumerate}
\item Each $\gamma_i$ is a permutation in $S_{2k}$ for some fixed $k > 0$.

\item $\gamma_1$ is a fixed-point-free involution.

\item $\gamma_0$ contains no cycles of length 1 or 2.

\item $\gamma_2 = \gamma_0^{-1}\circ\gamma_1$, so strictly speaking, is not a necessary part of the definition of the ribbon graph.

\item $b: \cycles{\gamma_2} \rightarrow \{1, 2, \ldots, \FiniteCount{\cycles{\gamma_2}} \}$ is a bijection.

\item The group generated by $\gamma_0$ and $\gamma_1$ acts transitively on $\nset{2k}$.

\end{enumerate}

\end{definition}

The map $b$ is called the boundary labeling of the graph, which will become clear in what follows. We also have the numbers $n = \FiniteCount{\cycles{\gamma_2}}$, 
$e = \FiniteCount{\cycles{\gamma_1}}$ and $v = \FiniteCount{\cycles{\gamma_0}}$.
The \emph{type} of the ribbon graph is the pair  $(g, n)$ where
$$
  g = 1 - \frac{1}{2}(v-e + n).
$$

To associate the above definition with an actual graph, we identify
$\cycles{\gamma_0}$ with the set of vertices of our graph,
$\cycles{\gamma_1}$ with the set of edges and $\cycles{\gamma_2}$
with the set of boundary paths. In particular, we take
$\FiniteCount{\cycles{\gamma_0}}$ vertices and to each vertex we
attach a number of half-edges equal to the length of the corresponding
cycle in $\gamma_0$. Each vertex can be cyclically ordered by
$\gamma_0$. The half-edges are glued to each other by using
$\gamma_1$. The construction
of the ribbon graph from the permutations is illustrated in
Figure~\ref{fig:RibbonGraphConstruction}.
 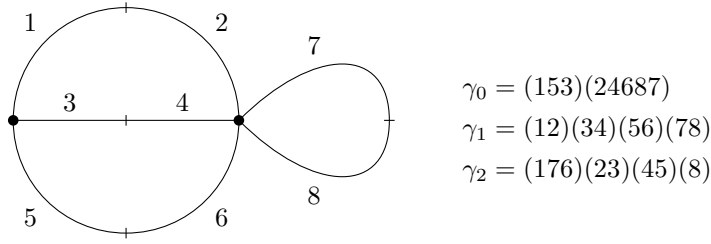
\begin{figure}
  \begin{tikzpicture}
	\draw (0,0) circle (1.5) (-1.5,0) -- (1.5,0);
	\draw (1.5,0) .. controls (2.5, 1.0) and (3.5, 1.0) .. (3.5, 0) --
	(3.5,0) .. controls (3.5, -1.0) and (2.5, -1.0) .. (1.5,0);
	\fill (-1.5,0) circle (2pt)  (1.5, 0) circle (2pt);
	\draw (45:1.5) node[anchor=south west] {$2$};
	\draw
	  (135:1.5) node[anchor=south east] {1}
	  (225:1.5) node[anchor=north east] {5}
	  (-45:1.5) node[anchor=north west] {6}
	  (-0.75, 0) node[anchor=south] {3}
	  (0.75, 0) node[anchor=south] {4}
	  (2.5, 1) node {7}
	  (2.5, -1) node {8};
	\draw (0,1.5) +(0,2pt) -- +(0,-2pt)
	  (0,-1.5) +(0,2pt) -- +(0,-2pt)
	  (0,0) +(0,2pt) -- +(0,-2pt)
	  (3.5,0) +(-2pt,0) -- +(2pt,0);
	\draw[xshift=4cm]
	  node[right, text width=4cm]
	  {
		\begin{align*}
		  \gamma_0 &= (153)(24687) \\
		  \gamma_1 &= (12)(34)(56)(78) \\
		  \gamma_2 &= (176)(23)(45)(8)
		\end{align*}
	  };
  \end{tikzpicture}
  \caption{Constructing a ribbon graph from half-edge permutations.}
\label{fig:RibbonGraphConstruction}
\end{figure}

Note that a ribbon graph constructed in this way has its half-edges
labeled; however, we do not wish to distinguish ribbon graphs which
only differ by their half-edge labelings. This motivates the notion of
equivalence of ribbon graphs: Two ribbon graphs $(\gamma_0,
\gamma_1, b)$ and $(\gamma'_0, \gamma'_1, b')$ are \emph{equivalent} if there
is a bijection $\alpha: \nset{2k} \rightarrow \nset{2k}$ such that
$\gamma'_i\circ\alpha = \alpha\circ\gamma_i$, and $b = b'\circ\alpha$.

One can, in a canonical way, construct an oriented surface from a
ribbon graph by replacing each vertex neighborhood with an oriented
disk, then using the edges to attach the disks to each other by
ribbons, making sure to preserve the orientation at each vertex.
Figure~\ref{fig:RibbonGraphExamples} illustrates two ribbon graphs
with their associated surfaces. It is straightforward to verify
that the surface associated to a given ribbon graph has genus $g$ and
$n$ holes, which explains the definition of the type of a graph. Note
that condition (6) in the definition forces the graph (and hence
surface) to be connected. There are circumstances when a disconnected
ribbon graph is allowed, but the changes to the theory are minor and
easily worked out.

In what follows, if $j\in \nset{2k}$ then the vertex incident to the
half-edge $j$ is denoted $[j]_0$. This can also be thought of as the
cycle of $\gamma_0$ which contains $j$. Similarly, the edge containing
$j$ is denoted $[j]_1$ while the corresponding boundary component is
$[j]_2$. We see that the valence (or degree) of a vertex (number of half-edges
incident to it) equals the size of its $\gamma_0$ orbit. In particular,
condition (3) requires that a ribbon graph has no 1- or 2-valent vertices.

We define $\Graphs_{g,n}$ to be the set of all equivalence classes of
ribbon graphs of type $(g,n)$. Because of the degree restriction on
vertices from condition (3), there is an upper bound of $12g-12 + 6n$
on the number of half-edges of a graph, realized exactly when the
graph is \emph{trivalent} -- all vertices having degree 3. As a
result, there are a finite number of equivalence classes of graphs of
a fixed type. Note that, in general, a ribbon graph
$G \in \Graphs_{g,n}$ may have automorphisms (self-equivalences)
and we let $\Aut(G)$ denote the automorphism group of $G$. For
example, $\Graphs_{1,1}$ consists of two graphs, as pictured in
Figure~\ref{fig:1-1Graphs}, with automorphism groups $\Aut(G_1) = \ZZ_6$
and $\Aut(G_2) = \ZZ_4$,
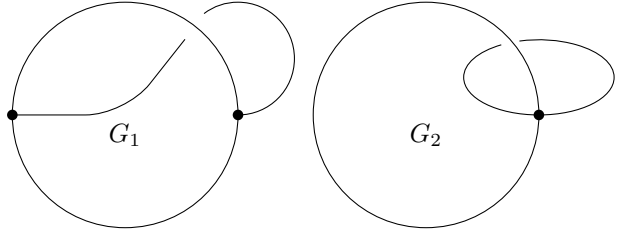
\begin{figure}
  \begin{tikzpicture}
	\draw (2, 0) circle (1.5cm);
	\draw[rounded corners=15pt] (3.5, 0.75) +(-90:0.75) arc
	  (-90:127:0.75)
	   (3.5, 0.75) +(160:0.75) -- (2, 0) -- (0.5, 0)
	   (2, 0) node[anchor=north] {$G_1$};
	\fill (0.5, 0) circle (2pt);
	\fill (3.5, 0) circle (2pt);
	\draw[xshift=4cm] (2, 0) circle (1.5cm)
	 (3.5, 0) arc (270:120:1cm and 0.5cm)
	 (3.5, 0) arc (-90:105:1cm and 0.5cm)
	 (2, 0) node[anchor=north] {$G_2$};
	 \fill[xshift=4cm] (3.5, 0) circle (2pt);
  \end{tikzpicture}
\caption{Set of all ribbon graphs of type $(1,1)$.}
\label{fig:1-1Graphs}
\end{figure}
while $\Graphs_{0,3}$ consists of seven distinct graphs, presented in
Figure~\ref{fig:0-3Graphs}, all with
trivial automorphism groups. Note that the graphs have non-trivial
automorphisms which permute the boundaries, which reduce the number of
distinct boundary labelings.
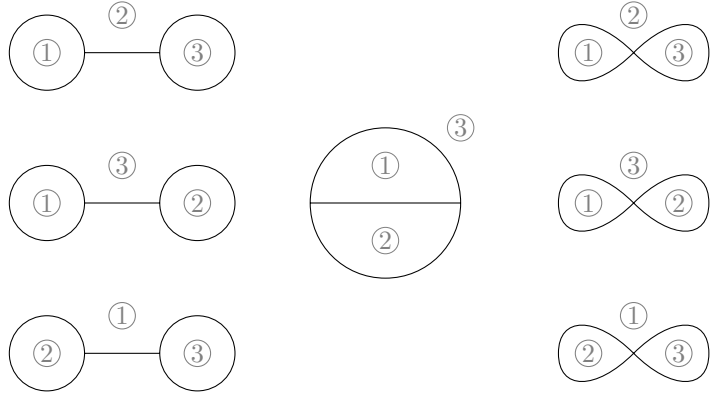
\begin{figure}
  \begin{tikzpicture}
	\draw (0,0) circle (1cm)
	  (-1, 0) -- (1, 0);
	\draw[ultra thin, gray]
	(0, 0.5) circle (0.5em) node {1} 
	  (0, -0.5) circle (0.5em) node {2}
	  (1,1) circle (0.5em) node {3};
	\draw[xshift=-3.5cm] (-1, 0) circle (0.5cm)
	  (1, 0) circle (0.5cm)
	  (-0.5, 0) -- (0.5, 0);
	\draw[xshift=-3.5cm, gray]
	  (-1,0) circle (0.5em) node {1}
	  (1,0) circle (0.5em) node{2}
	  (0,0.5) circle (0.5em) node {3};
	\draw[xshift=-3.5cm, yshift=2cm] (-1, 0) circle (0.5cm)
	  (1, 0) circle (0.5cm)
	  (-0.5, 0) -- (0.5, 0);
	\draw[xshift=-3.5cm, yshift=2cm, gray]
	  (-1,0) circle (0.5em) node {1}
	  (1,0) circle (0.5em) node{3}
	  (0,0.5) circle (0.5em) node {2};
	\draw[xshift=-3.5cm, yshift=-2cm] (-1, 0) circle (0.5cm)
	  (1, 0) circle (0.5cm)
	  (-0.5, 0) -- (0.5, 0);
	\draw[xshift=-3.5cm, yshift=-2cm, gray]
	  (-1,0) circle (0.5em) node {2}
	  (1,0) circle (0.5em) node{3}
	  (0,0.5) circle (0.5em) node {1};
	\draw[xshift=3.3cm, yshift=0cm] (0,0) .. controls (0.5, 0.5) and (1, 0.5) ..
	  (1,0) -- (1,0) .. controls (1, -0.5) and (0.5, -0.5) .. (0,0)
	  (0,0) .. controls (-0.5, 0.5) and (-1, 0.5) .. (-1, 0) --
	  (-1, 0) .. controls (-1, -0.5) and (-0.5, -0.5) .. (0,0);
	\draw[xshift=3.3cm, yshift=0cm, gray] 
	  (-0.6, 0) circle (0.5em) node {1}
	  (0.6, 0) circle (0.5em) node {2}
	  (0, 0.5) circle (0.5em) node {3};
	\draw[xshift=3.3cm, yshift=2cm] (0,0) .. controls (0.5, 0.5) and (1, 0.5) ..
	  (1,0) -- (1,0) .. controls (1, -0.5) and (0.5, -0.5) .. (0,0)
	  (0,0) .. controls (-0.5, 0.5) and (-1, 0.5) .. (-1, 0) --
	  (-1, 0) .. controls (-1, -0.5) and (-0.5, -0.5) .. (0,0);
	\draw[xshift=3.3cm, yshift=2cm, gray] 
	  (-0.6, 0) circle (0.5em) node {1}
	  (0.6, 0) circle (0.5em) node {3}
	  (0, 0.5) circle (0.5em) node {2};
	\draw[xshift=3.3cm, yshift=-2cm] (0,0) .. controls (0.5, 0.5) and (1, 0.5) ..
	  (1,0) -- (1,0) .. controls (1, -0.5) and (0.5, -0.5) .. (0,0)
	  (0,0) .. controls (-0.5, 0.5) and (-1, 0.5) .. (-1, 0) --
	  (-1, 0) .. controls (-1, -0.5) and (-0.5, -0.5) .. (0,0);
	\draw[xshift=3.3cm, yshift=-2cm, gray] 
	  (-0.6, 0) circle (0.5em) node {2}
	  (0.6, 0) circle (0.5em) node {3}
	  (0, 0.5) circle (0.5em) node {1};
  \end{tikzpicture}
  \caption{Set of all ribbon graphs of type (0,3). Boundary labelings are
  indicated by the circled numbers.}
  \label{fig:0-3Graphs}
\end{figure}

A metric on a ribbon graph $G = (\gamma_0, \gamma_1, b)$ is a           
function $\ell: \cycles{\gamma_1} \rightarrow \RR_+$, from the          
set of edges to the positive reals. One can think of a metric as        
determining the length of each edge of a graph. Note that if $e =       
\FiniteCount{\cycles{\gamma_1}}$ is the number of edges of $G$, then    
an element of $\RR_+^{e}$ determines a metric on $G$. If $G$ has        
nontrivial automorphisms, they act nontrivially on $\RR_+^{e}$ by       
permuting the coordinates. Hence, we see that the set of all metrics on 
a graph is naturally identified with                                    
$$\Met(G) = \quotient{\RR_+^{e}}{\Aut(G)},
$$ 
and we define the ribbon graph complex of type $(g,n)$ by
$$
\RG_{g,n} = \bigsqcup_{G\in \Graphs_{g,n}} \Met(G).
$$

The ribbon graph complex can be given a topology by considering edge
collapsing: taking the limit of an edge length to 0, for any non-loop
edge, results in a ribbon
graph of the same type with the corresponding edge contracted. The resulting set of
metric ribbon graphs is glued to the face of the metric set of the
larger graph.  
The resulting topological
space has the structure of a connected differentiable orbifold of
dimension $6g - 6 + 3n$ \cite{MR928904, 1998math.ph.11024M}.

Given a metric ribbon graph, one can assign perimeters to each boundary
of the graph by adding together the lengths of all edges which appear on
the boundary. Note that, in general, each edge appears on two boundaries
(each half edge is part of a boundary), so that it is possible for an
edge to contribute twice to a perimeter. We denote the perimeter map
$$
 p:\RG_{g,n} \rightarrow \RR_+^n.
$$
 If $L_{\nset{n}}$ denotes the vector $(L_1, \ldots, L_n) \in \RR_+^n$
 then we define
 $$
  \RG_{g,n}(L_{\nset{n}}) = p^{-1}(L_{\nset{n}}).
 $$
 In other words, it is the set of metric ribbon graphs with fixed boundary lengths.
 
 In general, if $\Gamma \in \RG_{g,n}$ is a metric ribbon graph with half edge $i\in\nset{2k}$, we will denote the length of the edge $[i]_1$ by $\ell(\Gamma, i)$. If the graph is clear from the context we will use $\ell_i = \ell(\Gamma, i)$. In addition, the ribbon graph underlying $\Gamma$ will be denoted by $\FiniteCount{\Gamma}$.  We can think of the $\ell_i$'s as a set of functions (or local coordinates if we choose one $i$ for each edge) defined on $\Met(\Gamma)$.
 
 Let $d(i)$ denote the degree of the vertex $[i]_0$. To each half-edge $i$ we have the vector field
 $$
  \tau_i = \sum_{j=1}^{d(i)-1} (-1)^j \frac{\partial}{\partial\ell_{\gamma_0^ji}}.
 $$
 We also define vector fields assigned to each edge
 $$
 T_i = T_{\gamma_1i} = \tau_i + \tau_{\gamma_1i}.
 $$

In addition to the orbit notation $[i]_j$ described above for vertices, edges and boundaries of a ribbon graph, we also introduce the following edge-length notation: If boundary $k$ contains $m_k$ half-edges we label the lengths of those edges by $\ell_1^{[k]}, \ldots, \ell_{m_k}^{[k]}$. Note that the total ordering of the edges must preserve the inherent cyclic ordering of the boundary, but a choice has been made in creating this list (i.e. choosing a distinguished starting edge out of the cyclically ordered boundary edges). 
  
 Following Kontsevich \cite{MR1171758}, we construct $n$ 2-forms on the ribbon graph complex (one for each boundary) by
 $$
  \omega_k = \sum_{i=1}^{m_k - 1}\sum_{j=i+1}^{m_k}d\ell_i^{[k]}\wedge d\ell_j^{[k]},
 $$
 then set
 $$
 \Omega = \frac{1}{2} \sum_{k=1}^n \omega_k.
 $$

 Note that $\Omega$ is not invariant under changes in the choices of total ordering at each boundary. However, the difference is always an exact form with
 $$
 \Omega - \Omega' = \sum_{i=1}^{n} a_i dp_i
 $$
 where $a_i$ are constants. Hence $\left. \Omega\right|_{\RG_{g,n}(L_{\nset{n}})}$ is well-defined. Moreover, Kontsevich \cite{MR1171758} proved that it is non-degenerate when restricted to cells corresponding to graphs with no even-valent vertices.
 
 We are led to define
 $$
 \Vol_{g,n}(L_{\nset{n}}) = \int_{\RG_{g,n}(L_{\nset{n}})}e^{\Omega} =  \int_{\RG_{g,n}(L_{\nset{n}})} \frac{1}{d!}\Omega^d,
 $$
 where $d = 3g-3 + n$.

 In general, the dimension of $\RG_{g,n}$ is equal to $6g - 6 + 3n$, which corresponds with the number of edges in a \emph{trivalent} ribbon graph (all vertices have degree 3). Because they play a special role in what follows, we denote $\RG^3_{g,n}$ to be the space of trivalent metric ribbon graphs.
Although, strictly speaking, $\Omega^d$ is not a volume form, being
degenerate on ribbon graphs with even-valent vertices, it is
non-degenerate on the top-dimension strata
$\RG_{g,n}^3(L_{\nset{n}})$. Since integration over a set of measure 0
does not contribute, the volume is well-defined.
 \subsection{Intersection theory on $\CompactModuli_{g,n}$}
The primary motivation for studying the ribbon graph complex is
because of its close connection to the moduli space of curves
$\Moduli_{g,n}$. In fact, a result attributed to Mumford, Thurston and
Harer \cite{MR963064}
states that $\Moduli_{g,n}\times \Reals_+^n$ is diffeomorphic (in the
sense of orbifolds) to $\RG_{g,n}$. This result follows by examining
foliations from Strebel differentials on surfaces. A similar result
was proven by Bowditch-Epstein \cite{MR935529}, and independently by Penner
\cite{MR919235} using hyperbolic geometry.

These results were utilized by Kontsevich \cite{MR1171758} to great effect in
his celebrated proof of the Witten conjecture \cite{MR1144529}. By careful
analysis of degenerating ribbon graphs, he was able to use the ribbon
graph complex in calculating intersection numbers over the
Deligne-Mumford compactification of moduli space
$\CompactModuli_{g,n}$. To be precise there is a compactification of
the ribbon graph complex $\CompactRG_{g,n}(L)$ on which the symplectic form
$\Omega$ extends and a map
\begin{equation*}
  q: \CompactModuli_{g,n} \rightarrow \CompactRG_{g,n}(L)
\end{equation*}
for which $q^*\Omega$ represents the sum tautological classes
$\frac{1}{2}(L_1^2\psi_1 + \cdots + L_n^2 \psi_n)$.

Hence, one interpretation of the symplectic volume discussed in the
previous section is that it encodes all intersections of
$\psi$-classes
on $\CompactModuli_{g,n}$. In fact,
\begin{equation}
  \Vol_{g,n}(L_N) = \sum_{k_1 + \cdots + k_n = d}
  \prod_{j=1}^{n}\frac{L_j^{2k_j}}{2^{k_j}k_j!}
  \int_{\CompactModuli_{g,n}} \psi_1^{k_1}\cdots\psi_n^{k_n}.
  \label{eq:VolumePsiRelation}
\end{equation}

 \subsection{Eynard-Orantin topological recursion}
 The topological recursion formula presented in
 Section~\ref{sect:RecursionFormula} fits
 into the framework developed by Eynard and Orantin \cite{MR2346575}, which
 we now proceed to outline. 

 Consider a plane algebraic curve $C$ specified by a
 polynomial equation
 \begin{align*}
   C^o &= \left\{ (x,y)\in \CC^2\, |\, P(x, y) = 0 \right\} \\
   C &= \overline{C^o}.
 \end{align*}
 It is convenient to think of $x$ and $y$ as a choice of two
 meromorphic functions on $C$. In other words, given a local
 coordinate $z\in C$ we have
 \begin{align*}
   x &= x(z) \\
   y &= y(z).
 \end{align*}
We require the projection of $C$ onto the $x$-axis to be
\emph{generic}: branch points are isolated and degree at most two
(simply ramified).

Note that the theory developed by Eynard and Orantin applies in a
wider setting than presented here, but restricting $x$ and $y$ to be
rational functions is more than sufficient for our needs, and makes
the theory somewhat simpler. In what follows, we make the additional
(unnecessary) assumption that $C=\PP^1$, with global
coordinate $z$. 

To the data of a spectral curve, one can associate an infinite tower
of symmetric multilinear meromorphic differentials $\mathcal{W}_{g,n}(z_1,
\ldots, z_n) = W_{g,n}(z_1, \ldots,
z_n)dz_1 \otimes \cdots \otimes dz_n$ defined on $\Sym^nC$. They are
constructed in a recursive manner, by performing residue computations
around the branch points of the $x$-projection.

In particular, the base cases of the recursion
are
\begin{align*}
  \mathcal{W}_{0,1}(z) &= 0 \\
  \mathcal{W}_{0,2}(z_1, z_2) &= \frac{dz_1\otimes dz_2}{(z_1 - z_2)^2}.
\end{align*}
$\mathcal{W}_{0,2}$ is the \emph{Cauchy differentiation kernel},
defined by
the property that for any function $f: C \rightarrow \PP^1$
\begin{equation*}
  f'(z)dz = \Res_{\zeta \rightarrow z} f(\zeta)\mathcal{W}_{0,2}(\zeta, z).  
\end{equation*}
Note that the differentiation kernel is also referred to as the \emph{Bergmann
kernel} in the literature \cite{MR2346575}. In addition, if $C$ has genus
greater than zero, then the $\mathcal{A}$-cycle integrals of the
kernel need to be specified in order
to have a unique differential form.

A few additional constructions are necessary to derive the
higher-order invariants. The first is a notion of \emph{conjugate
point:} Let $a_1, \ldots, a_k$ be the branch points of the projection
of $C$ onto the $x$-axis. If $z\in C$ is sufficiently close to a
branch point $a_i$ then there is a unique point $\bar{z} \neq z$ with
the same $x$-projection as $z$ (due to the fact that all branch points
are simple). Note that unlike complex conjugation, the locally defined
involution $z \mapsto \bar{z}$ is holomorphic.

We also make use of the \emph{Eynard kernel}, defined as
\begin{equation*}
  E_i(z_1, z_2) = \frac{1}{2} \int_{z_1}^{\bar{z}_1} W_{0,2}(\zeta,
  z_2)d\zeta  \frac{dz_2}{\bigr(y(z_1) - y(\bar{z}_1)\bigr)dx(z_1)},
\end{equation*}
where $E_i$ is defined locally around the branch point $a_i$ (from
which the conjugation operation is defined) and  the operator on 
differential forms $\frac{1}{dx(z)}$ means
contraction with the vector field 
\begin{equation*}
  \frac{1}{\frac{dx}{dz}}\frac{d}{dz}.
\end{equation*}

Then, the higher order Eynard-Orantin invariants are defined by the
recursion formula
\begin{multline*}
  \mathcal{W}_{g,n+1}(z, z_{\nset{n}}) = \sum_i \Res_{\zeta\rightarrow a_i}
  E_i(z, \zeta) \Bigl[
  \mathcal{W}_{g-1, n+2}(\zeta, \bar{\zeta}, z_{\nset{n}}) \\
  + \sum_{g_1 + g_2 = g} \sum_{\mathcal{I} \sqcup \mathcal{J} =
  \nset{n}} \mathcal{W}_{g_1, \FiniteCount{\mathcal{I}} + 1}(\zeta, z_{\mathcal{I}})
  \mathcal{W}_{g_2, \FiniteCount{\mathcal{J}} + 1}(\bar{\zeta},
  z_{\mathcal{J}})
  \Bigr].
\end{multline*}

The Eynard-Orantin invariants have appeared in a broad array of
seemingly unconnected mathematics. Some
highlights include
\begin{enumerate}

\item The correlation functions for $y = \sin(\sqrt{x})$ are related to (via the Laplace transform) the Weil-Petersson symplectic volumes for moduli spaces of bordered Riemann surfaces \cite{eynard-orantin-2007}, and the recursion formula is equivalent to the recursion formula first discovered by Mirzakhani \cite{MR2264808,MR2257394} in the context of hyperbolic geometry.

\item The recursion for intersection numbers of mixed $\psi$ and
  $\kappa_1$ classes originally discovered by Mulase and Safnuk
  \cite{MR2379144}, and then extended to arbitrary $\kappa$ classes by Liu and Xu \cite{MR2357474, MR2348851, MR2482127} were put into the framework of topological recursion by Eynard \cite{eynard-2007}.

\item Topological recursion can be used to calculate the generating function enumerating partitions with the Plancheral measure \cite{MR2439683, MR2569934}.

\item Correlation functions for the curve of the mirror dual to a 3-dimensional toric Calabi-Yau manifold are conjectured to generate the Gromov-Witten potential of the manifold. \cite{MR2480744,Eynard:vn,Eynard:yq}.

\item The Lambert curve $x = ye^{-x}$ gives the generating functions
  for Hurwitz numbers  \cite{borot-2009,MR2483754,eynard-2009}, giving
  a positive resolution to a conjecture raised by Bouchard and
  Mari\~no \cite{MR2483754}.

  \item The curve $x = y + 1/y$ was shown by Norbury \cite{norbury-2009} to compute the
  number of lattice points in the moduli space of curves. Refer to
  \cite{2010arXiv1009.2055C} for a related construction.
\end{enumerate}


The simplest non-trivial example of a spectral curve is the Airy curve
\begin{align*}
  x &= \frac{1}{2}z^2 \\
  y &= z,
\end{align*}
which is a rational curve, with global coordinate $z$. There is a single branch point at $(0,0)$, with
a globally defined involution $z \mapsto -z$.

The Cauchy differentiation kernel for the Riemann sphere is given by
\begin{equation*}
  \mathcal{W}_{0,2}(z_1, z_2) = \frac{dz_1 \otimes dz_2}{(z_1 -
  z_2)^2},
\end{equation*}
while the Eynard kernel at the unique branch point is
\begin{equation*}
  E(z_1, z_2) = \frac{1}{z_1^2 - z_2^2}\frac{dz_2}{2z_1dz_1}.
\end{equation*}
This yields a recursion formula
\begin{equation}
  \begin{split}
	\mathcal{W}_{g,n}(z_1, \ldots, z_n) & = 
	\Res_{\zeta\rightarrow 0} \frac{dz_1}{2\zeta(\zeta^2 - z_1^2)
	d\zeta} \biggl( \mathcal{W}_{g-1, n+1}(\zeta, -\zeta, z_2, \ldots,
	z_n) \\
	& \quad + \sum_{\substack{g_1 + g_2 = g \\ \mathcal{I} \sqcup
	\mathcal{J} = \nset{n} \setminus \{ 1 \} }} \mathcal{W}_{g_1,
	n_1}(\zeta, z_{\mathcal{I}}) \mathcal{W}_{g_2, n_2}(-\zeta,
	z_{\mathcal{J}})
	\biggr)
  \end{split}
  \label{eq:AiryResidueRecursion}
\end{equation}

Applying the Eynard-Orantin recursion to the first few cases gives
\begin{align*}
  \mathcal{W}_{0,3}(z_1, z_2, z_3) &= \frac{dz_1\otimes dz_2\otimes
  dz_3}{z_1^2 z_2^2 z_3^2} \\
  \mathcal{W}_{1,1}(z) &= \frac{dz}{8z^4} \\
  \mathcal{W}_{0,4}(z_1, z_2, z_3, z_4) &= 3 
  \left( \frac{1}{z_1^2} + \frac{1}{z_2^2} + \frac{1}{z_3^2} +
  \frac{1}{z_4^2} \right)\prod_{i=1}^{4} \frac{dz_i}{z_i^2}.
\end{align*}


 \subsection{Symplectic geometry}
 \label{sect:SymplecticReduction}
The goal of the paper is to calculate the symplectic volume of the
ribbon graph complex. The technique presented relies on several
standard constructions from symplectic geometry which we now review.

The pair $(M, \omega)$ is a symplectic manifold if $M$ is a smooth
$2n$-manifold and $\omega$ is a closed, non-degenerate 2-form on $M$.
The non-degeneracy condition of $\omega$ forces $M$ to be
even-dimensional. In general, if $(M, \omega)$ is a symplectic
manifold then the top-dimension form $\frac{1}{n}\omega^n$ is
everywhere non-degenerate, and therefore is a volume form. In the case
that $M$ is compact (or $\omega^n$ is integrable) we define
\begin{equation*}
  \Vol(M, \omega) = \int_M \frac{1}{n!}\omega^n.
\end{equation*}
Note in particular that the symplectic volume depends on the form
$\omega$; however when $M$ is compact the volume is an invariant of
the cohomology class of $\omega$.

Suppose that $(M, \omega)$ has a $k$-torus symmetry. In particular,
there is a $k$-parameter group of diffeomorphisms
\begin{equation*}
  F_{(t_1, \ldots, t_k)}: M \rightarrow M
\end{equation*}
satisfying the following conditions:
\begin{enumerate}
  \item $F_{(t_1, \ldots, t_k)}$ is a symplectomorphism for all
	$t = (t_1, \ldots, t_k)\in \RR^{k}$, i.e. $F_{t}^*\omega = \omega$.

  \item For all $t, t' \in \RR^k$, $F_t\circ F_t' = F_{t + t'}$.

  \item There exists $c \in \RR_+^k$ with $F_{t + c} = F_t$ for all
	$t \in \RR^k$. The constant $c_j$ is the \emph{period} or
	circumference of the $j$-th component of the torus
	action.
\end{enumerate}

The symplectic torus action encoded by $F$ has $k$ commuting vector
fields denoted $X_1, \ldots, X_k$, constructed by taking derivatives
of $F$:
\begin{equation*}
  X_j(x) = \left. \frac{\partial F_t (x)}{\partial t_j}\right|_{t=0}
\end{equation*}

A symplectic torus action is called \emph{Hamiltonian} if there is, in
addition to the above conditions, a
map $\mu: M \rightarrow \RR^k$ (called the moment map) 
satisfying the duality condition
\begin{equation*}
  \iota_{X_j}\omega = d\mu_j.
\end{equation*}
Note that $\iota_X$ is the contraction operator, taking the $q$-form
$\alpha$ to the $(q-1)$-form such that for any collection of vector
fields $Y_1, \ldots, Y_{q-1}$
\begin{equation*}
  \iota_X\alpha(Y_1, \ldots, Y_{q-1}) = \alpha(X, Y_1, \ldots,
  Y_{q-1}).
\end{equation*}

A key property of the moment map is that the torus action preserves
level sets: $F_t\mu^{-1}(a) \subset \mu^{-1}(a)$ for all $a, t\in
\RR^k$. In addition, in many situations the quotient of a level set
by the torus is still a manifold. We denote the quotient space
\begin{equation*}
  M_a = \mu^{-1}(a) / \Torus^k.
\end{equation*}
In fact, it will be a symplectic
manifold with a canonical symplectic form $\omega_a$ induced from the
original symplectic structure. To be precise, let $q: \mu^{-1}(a)
\rightarrow M_a$ be the quotient map. If $Y_1, Y_2$ are two tangent
vectors on $M_a$ we choose arbitrary lifts $\tilde{Y}_i$ (i.e.
$q_*\tilde{Y}_i = Y_i$), and define
\begin{equation*}
  \omega_a(Y_1, Y_2) = \omega(\tilde{Y}_1, \tilde{Y}_2).
\end{equation*}
One can check that $\omega_a$ is well-defined, closed and
non-degenerate. 

The above construction is called \emph{symplectic reduction}. The
relevance in the present situation is its applications in volume
calculations. Let $D\subset \RR^k$ be the image of the moment map. By
a theorem of Guillemin and Sternberg \cite{MR664117}, $D$ is a convex polytope. One can
define the Duistermaat-Heckman measure on $D$ by considering the
volume form
\begin{equation*}
	\Vol(M_x)dx_1\cdots dx_k.
\end{equation*}
In fact, we have \cite{MR674406} 
\begin{equation}
	\Vol(M) = \int_D \prod c_i \Vol(M_x) dx_1 \cdots dx_k.
	\label{eq:ReductionVolumeFormula}
\end{equation}

In the present context, however, the ribbon graph complex does not admit
a global circle action. To circumvent this difficulty, we construct a locally 
finite cover
$\{U_i\}$ and corresponding partition of unity $\{\phi_i\}$.
We assume that each symplectic manifold $(U_i, \omega)$ has a
Hamiltonian circle action with moment map $\mu_i: U_i \rightarrow
\mathbb{R}$. Note that the following discussion can be trivially
extended to torus actions, but for ease of notation we suppress such
generalities. The key assumption we are making is that
\textbf{$\phi_i$ is equivariant with respect to the circle action on $U_i$}.
Equivalently, $\phi_i$ is constant on the level sets $\mu_i^{-1}(x)$.
Hence there is a function $f_i: \mathbb{R} \rightarrow \mathbb{R}$
with $f_i\circ \mu_i = \phi_i$.

The goal is to calculate the symplectic volume
\begin{equation*}
  \int_M \frac{1}{n!} \omega^n
\end{equation*}
which we first write as a sum of integrals using the partition of
unity:
\begin{equation*}
  \int_M \frac{1}{n!} \omega^n = \sum_i \int_{U_i} \frac{\phi_i}{n!}
  \omega^n.
\end{equation*}

In order to calculate the integral over $U_i$ we utilize the
Hamiltonian circle action. We let $V_i(x) = \mu_i^{-1}(x) / S^1$ be
the symplectic quotient with induced symplectic form $\omega_i(x)$.
Note that at this stage, the partition function $\phi_i$ is not part
of the construction. We let $\Vol_i(x)$ be the volume of the quotient.

Recall the Duistermaat-Heckmann measure on $\mathbb{R}$: If $A\subset
\mathbb{R}$ is any measurable
subset we define
\begin{equation*}
  m_i^{DH}(A) = \int_{\mu_i^{-1}(A)} \frac{\omega^n}{n!}.
\end{equation*}
In particular, we can integrate the function $f_i$ with respect to
this measure to get
\begin{equation*}
  \int_{\mathbb{R}} f_i(x) m_i^{DH}(x) =
  \int_{U_i}\frac{\mu_i^*(f_i)}{n!} \omega^n
  = \int_{U_i} \frac{\phi_i}{n!} \omega^n.
\end{equation*}

To complete the calculation, we relate the Duistermaat-Heckmann
measure to ordinary Lebesque measure with the Radon-Nikodym
derivative. According to Duistermaat and Heckmann this derivative is
equal to $\Vol_i(x)$ times the circumference of the circle action.
Hence
\begin{equation*}
  \int_{U_i} \frac{\phi_i}{n!} \omega^n = \int_{\mu_i(U_i)} f_i(x)
  \theta_i(x) \Vol_i(x)\,dx,
\end{equation*}
where $\theta_i(x)$ is the circumference of the circle action at level
$x$. Note that in the present work, this circumference is equal
to $x$, hence the choice of coordinates is analogous to polar
coordinates, whereas cartesian coordinates would have constant
circumference.
\section{Local Structure}
\label{sect:LocalStructure}
In this section we construct locally defined Hamiltonian torus actions
on the ribbon graph complex. A careful analysis of the domain on which
the group action is defined allows for a partition of unity
subordinate to the open cover induced by the various domains. As a
consequence, one can derive a formula for the volume of the ribbon
graph complex by using the symplectic reduction techniques outlined in
Section~\ref{sect:SymplecticReduction}. The symplectic quotients are
themselves ribbon graph complexes, involving graph types of less
complexity (where the complexity of a graph of type $(g,n)$ is
measured by $2g-2+n$). The result is a recursive formula for
calculating the symplectic volumes. 

Let $\Gamma$ be a trivalent metric ribbon graph. Given an edge $[i]_1$
we define the metric ribbon graph $\Gamma_{\hat{i}}$ obtained by
removing the edge $[i]_1$ from $\Gamma$ and straightening the
resultant 2-valent vertices into contiguous edges, as depicted in
Figure~\ref{fig:EdgeRemoval}. The edge lengths of $\Gamma_{\hat{i}}$ are inherited from $\Gamma$.
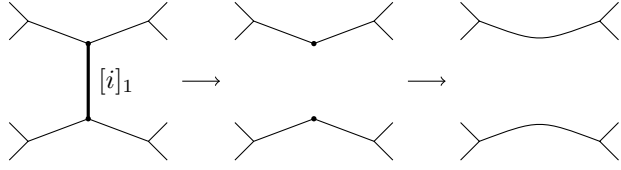
\begin{figure}
\begin{tikzpicture}
  \draw (0,0) -- ++(.8,-0.3) -- +(0.25,0.25)
	+(0,0.) -- +(0.25, -0.25)
	(0,0) -- ++(-0.8, -0.3) -- +(-0.25, 0.25)
  (-0.8,-0.3) -- +(-0.25, -0.25)
  (0,1) -- ++(.8, 0.3) -- + (0.25, 0.25)
  (0.8, 1.3) -- +(0.25, -0.25)
  (0,1) -- ++(-0.8, 0.3) -- + (-0.25, 0.25)
  (-0.8, 1.3) -- +(-0.25, -0.25)
  [->] (1.25, 0.5) -- (1.75, 0.5);
  \fill  (0,0) circle (1pt) (0,1) circle (1pt);
  \draw[very thick] (0,0) -- node[anchor=west] {$[i]_1$} +(0,1);
  \draw [xshift=3cm] (0,0) -- ++(.8,-0.3) -- +(0.25,0.25)
	+(0,0.) -- +(0.25, -0.25)
	(0,0) -- ++(-0.8, -0.3) -- +(-0.25, 0.25)
  (-0.8,-0.3) -- +(-0.25, -0.25)
  (0,1) -- ++(.8, 0.3) -- + (0.25, 0.25)
  (0.8, 1.3) -- +(0.25, -0.25)
  (0,1) -- ++(-0.8, 0.3) -- + (-0.25, 0.25)
  (-0.8, 1.3) -- +(-0.25, -0.25)
  [->] (1.25, 0.5) -- (1.75, 0.5);
  \fill [xshift=3cm] (0,0) circle (1pt) (0,1) circle (1pt);
  \draw [xshift=6cm] (-.8, -0.3) .. controls (0,0) .. (.8,-0.3) -- +(0.25,0.25)
	+(0,0.) -- +(0.25, -0.25)
	(-0.8, -0.3) -- +(-0.25, 0.25)
  (-0.8,-0.3) -- +(-0.25, -0.25)
  (-.8, 1.3) .. controls (0,1) .. (.8, 1.3) -- + (0.25, 0.25)
  (0.8, 1.3) -- +(0.25, -0.25)
  (-0.8, 1.3) -- + (-0.25, 0.25)
  (-0.8, 1.3) -- +(-0.25, -0.25);
\end{tikzpicture}
  \caption{Edge removal from a trivalent ribbon graph.}
\label{fig:EdgeRemoval}
\end{figure}

Note that there is an exception to the above operation in case $[i]_1$
adjoins (or is itself) a loop. $\Gamma_{\hat{i}}$ is defined by
removing the entire lollipop from $\Gamma$, as seen in
Figure~\ref{fig:LollipopRemoval}.
\begin{figure}
  \begin{tikzpicture}
  \draw (0,0) -- ++(.8,-0.3) -- +(0.25,0.25)
	+(0,0.) -- +(0.25, -0.25)
	(0,0) -- ++(-0.8, -0.3) -- +(-0.25, 0.25)
  (-0.8,-0.3) -- +(-0.25, -0.25)
  [->] (1.25, 0.5) -- (1.75, 0.5);
  \fill  (0,0) circle (1pt) +(0, 0.7) circle (1pt);
  \draw[very thick] (0,0) -- node[anchor=west] {$[i]_1$} +(0,0.7)
	(0, 1.2) circle (0.5);
  \draw [xshift=3cm] (-.8, -0.3) .. controls (0,0) .. (.8,-0.3) -- +(0.25,0.25)
	+(0,0.) -- +(0.25, -0.25)
	(-0.8, -0.3) -- +(-0.25, 0.25)
  (-0.8,-0.3) -- +(-0.25, -0.25);
  \end{tikzpicture}
\caption{Removing a lollipop.}
\label{fig:LollipopRemoval}
\end{figure}
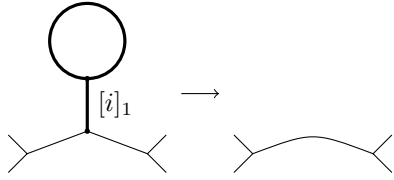

If we remember the locations of the deleted vertices, we have two
marked points on the boundary of $\Gamma_{\hat{i}}$. Let $m(\Gamma, i)$
denote the number of distinct boundary components on which the markings
appear (either 1 or 2). Rotating the marked points can be realized as
an $m$-torus orbit in $\RG_{g,n}$ (if $m=1$ the marked points must be
rotated in sync). We consider the lollipop removal case to also have 1
marked boundary ($m=1$), since rotation on a simple loop is a trivial
action. 
We call these rotations edge-twist deformations, as one
imagines twisting the edge $[i]_1$ around its connections to the
remainder of the graph. The
infinitesimal generators of these deformations are $T_i$ when it is a
circle action and the pair $(\tau_i, \tau_{\gamma_1 i})$ in the case
of a torus action. When edge $[i]_1$ forms a loop, the relevant vector
field is $T_{\gamma_0 i} + T_{\gamma_0^2 i}$ (exactly one of the two
terms is nonzero).

The set of all ribbon graphs obtained during one complete rotation of
edge $[i]_1$ is called the torus orbit of $(\Gamma, i)$, and denoted
$\Orbit(\Gamma, i)$. For any $\Gamma \in \RG_{g,n}^3(L_{\nset{n}})$
we consider the set
$$
 \TorusCover(\Gamma, i) = 
 \bigcup_{\tilde{\Gamma} \in \Met(\FiniteCount{\Gamma}; L_{\nset{n}})} 
 \Orbit(\tilde\Gamma, i) ,
$$
where $\Met(\FiniteCount{\Gamma}; L_{\nset{n}}) =
\RG_{g,n}(L_{\nset{n}}) \cap \Met(\FiniteCount{\Gamma})$. Note that
$\TorusCover(\Gamma, i) \subset \RG_{g,n}(L_{\nset{n}})$.

If we restrict attention to edges which are adjacent to the
first boundary (boundary label $1$) we still obtain a cover of the trivalent strata:
$$
\RG^3_{g,n}(L_{\nset{n}}) \subset \bigcup_{\Gamma \in
\RG^3_{g,n}(L_{\nset{n}})} \bigcup_{i : b(i) = 1} \TorusCover(\Gamma, i).
$$
In addition, each subset $\TorusCover(\Gamma, i)$ has a well-defined
function $f_{\Gamma, i}(\tilde{\Gamma})= \ell(\tilde{\Gamma}, i)$.

\begin{lemma}
 The collection of functions $\{ \frac{1}{L_1} f_{\Gamma, i}  \}$
 forms a partition of unity subordinate to the cover $\{
 \TorusCover(\Gamma, i)\, | \, \Gamma \in \RG^3_{g,n}(L_{\nset{n}}), b_{\Gamma}(i) = 1 \}$.
\end{lemma}
\begin{proof}
 This follows from the observation that the sum of edge lengths around
the first boundary equals, by definition, $L_1$. 
\end{proof}

We note that, by construction, each $\TorusCover(\Gamma, i)$ has a
globally defined torus action. The dimension of the torus is either 1
or 2, depending on the configuration of the vertices incident to edge
$[i]_1$, as discussed earlier.
\begin{lemma}
 The torus action on $\TorusCover(\Gamma, i)$ is Hamiltonian, with
moment map given by the period(s) of the action.
\end{lemma}
\begin{proof}
We must calculate the contraction of $\Omega$ by the vector
fields $\tau_i$ and $\tau_{\gamma_1 i}$ in the torus action case
and the vector field $T_i$ in the circle action case. Beginning
with the case of the circle action, refer to
Figure~\ref{fig:contractionLabels}
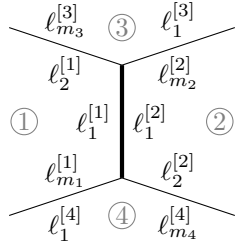
\begin{figure}
  \begin{tikzpicture}
	\draw 
	  (0,0) -- node[below] {$\ell_{m_4}^{[4]}$} 
	  node[above] {$\ell_{2}^{[2]}$} +(1.50, -0.5)
	  (0,0) -- node[above] {$\ell_{m_1}^{[1]}$}
	  node[below] {$\ell_{1}^{[4]}$} +(-1.5, -0.5);
	\draw[ultra thick]
	  (0,0) -- node[left] {$\ell_{1}^{[1]}$}
	  node[right] {$\ell_{1}^{[2]}$} (0,1.5);
	\draw
	  (0,1.5) -- node[above] {$\ell_{1}^{[3]}$} 
	  node[below] {$\ell_{m_2}^{[2]}$} +(1.5, 0.5)
	  +(0,0) -- node[above] {$\ell_{m_3}^{[3]}$}
	  node[below] {$\ell_{2}^{[1]}$} +(-1.5, 0.5);
	\draw[gray]
	  (-1.3, 0.75) node {1} circle (0.5em) 
	  (1.3, 0.75) node {2} circle (0.5em)
	  (0, -0.5) node {4} circle (0.5em)
	  (0, 2.0) node {3} circle (0.5em);
  \end{tikzpicture}
\caption{Edge labels used to calculate vector field contraction. Note
that edge $[i]_1$ is in bold.}
\label{fig:contractionLabels}
\end{figure}
for the notation used in what follows.

Note that the only terms in $\Omega$ which contribute are $\omega_i$
for $i=1, 2, 3, 4$. Without loss of generality, we may assume that edge
$[i]_1$ corresponds with $\ell_1^{[1]}$ and $\ell_1^{[2]}$, while edge
$[\gamma_0\gamma_1 i]_1$ corresponds with $\ell_1^{[3]}$ and
$[\gamma_0 i]_1$
corresponds with $\ell_1^{[4]}$. Under this labeling we have
\begin{align*}
T_i &= -\frac{\partial}{\partial \ell_{m_2}^{[2]}} + \frac{\partial}{\partial\ell_{2}^{[1]}}
	- \frac{\partial}{\partial\ell_{m_1}^{[1]}} + \frac{\partial}{\partial\ell_{2}^{[2]}} \\
	&= -\frac{\partial}{\partial\ell_{1}^{[3]}} + \frac{\partial}{\partial\ell_{m_3}^{[3]}}
	- \frac{\partial}{\partial\ell_{1}^{[4]}} + \frac{\partial}{\partial\ell_{m_4}^{[4]}}.
\end{align*}
 It is straightforward to calculate
\begin{align*}
\iota_{T_i}\omega_1 &= -d\ell_{1}^{[1]} + (d\ell_{3}^{[1]} + \cdots + d\ell_{m_1}^{[1]})
	+ (d\ell_{1}^{[1]} + \cdots + d\ell_{m_1 - 1}^{[1]}  ) \\
	&= 2dp_1 - 2d\ell_{1}^{[1]} - d\ell_{2}^{[1]} - d\ell_{m_1}^{[1]} \\
\iota_{T_i}\omega_2 &= 2dp_2 - 2d\ell_{1}^{[2]} - d\ell_{2}^{[2]} - d\ell_{m_2}^{[2]} \\
\iota_{T_i}\omega_3 &= -(d\ell_{2}^{[3]} + \cdots + d\ell_{m_3}^{[3]} ) 
	+ (-d\ell_{1}^{[3]} - \dots - d\ell_{m_3 - 1}^{[3]}) \\
	&= -2dp_3 + d\ell_{1}^{[3]} + d\ell_{m_3}^{[3]} \\
\iota_{T_i}\omega_4 &= -2dp_4 + d\ell_1^{[4]} + d\ell_{m_4}^{[4]}.
\end{align*}
Note that, although slightly more complicated, nothing fundamentally
changes in the above calculation if some of the boundaries happen to
agree.

Since $\ell_{2}^{[1]} = \ell_{m_3}^{[3]}$, $\ell_{m_1}^{[1]} =
\ell_{1}^{[4]}$, $\ell_2^{[2]}=\ell_{m_4}^{[4]}$, $\ell_{m_2}^{[2]} =
\ell_{1}^{[3]}$, and $\ell_{1}^{[1]} = \ell_i = \ell_{1}^{[2]}$ (being
different labels for the same edges) we have
\begin{equation*}
\iota_{T_i}\Omega = d(p_1 + p_2 - 2\ell_i) - dp_3 - dp_4.
\end{equation*}
When restricted to $\RG_{g,n}(L_{\nset{n}})$ we have $dp_k = 0$, and
observing that $p_1 + p_2 - 2\ell_i$ is the length of the circle around
which edge $i$ rotates completes the first part of the proof.

The torus action case occurs when edge $i$ has the same boundary on
either side. Refer to Figure~\ref{fig:torusContractionLabels}
\begin{figure}
  \begin{tikzpicture}
	\draw (0,0) +(0:1.5) -- 
	  node[above=1pt, right=1pt] {$\ell_{k-1}^{[1]}$}
	  node[below=6pt, left=-1pt] {$\ell_{1}^{[3]}$}
	  +(72:1.5) -- +(144:1.5)
	  -- +(216:1.5) -- +(288:1.5) -- 
	  node[below, right] {$\ell_{2}^{[1]}$} 
	  node[above=6pt, left=-1pt] {$\ell_{m_3}^{[3]}$} 
	  (1.5,0)
	  (4.5,0) +(36:1.5) -- +(108:1.5) --
	  node[above=6pt, left=-1pt] {$\ell_{k+1}^{[1]}$}
	  node[below=6pt, right=-1pt] {$\ell_{m_2}^{[2]}$}
	  +(180:1.5) --
	  node[below, left] {$\ell_{m_1}^{[1]}$}
	  node[above=6pt, right=-1pt] {$\ell_{1}^{[2]}$}
	  +(252:1.5) -- +(324:1.5) -- cycle;
	\draw[ultra thick] (1.5, 0) --
	  node[above] {$\ell_{k}^{[1]}$}
	  node[below] {$\ell_{1}^{[1]}$}
	  (3,0);
	\draw[gray] (0,0) +(72:1.5) -- +(72:0.8)
	  +(144:1.5) -- +(144:0.8) +(216:1.5) -- +(216:0.8)
	  +(288:1.5) -- +(288:0.8)
	  (4.5,0) +(36:1.5) -- +(36:0.8)  +(108:1.5) -- +(108:0.8)
	  +(252:1.5) -- +(252:0.8) +(324:1.5) -- +(324:0.8);
	\draw[gray] 
	(0,0) node {3} circle (0.5em)
	(4.5, 0) node {2} circle (0.5em)
	(2.25, 1.25) node {1} circle (0.5em)
	(2.25, -1.25) node {1} circle (0.5em);
  \end{tikzpicture}
\caption{Edge labels used to calculate vector field contraction. The
edge in bold is $[i]_1$.}
\label{fig:torusContractionLabels}
\end{figure}
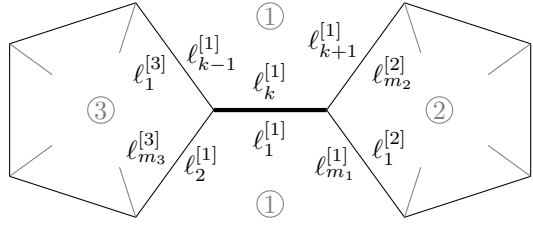
for the notation used in what follows.

When traversing boundary 1, we assume that $\ell_i = \ell_{1}^{[1]}$,
and note that edge $i$ divides the perimeter into two distinct regions
(which become the two distinct circles for the torus action). We label
$\ell_{k}^{[1]}$ as the second occurrence of $\ell_i$ in the
perimeter, which makes the two regions labeled by $\ell_{2}^{[1]}, 
\ldots, \ell_{k-1}^{[1]}$ and $\ell_{k+1}^{[1]}, \ldots, \ell_{m_1}^{[1]}$. 

The vector fields under this labeling are given by
\begin{align*}
  \tau_i &= \frac{\partial}{\partial \ell_{2}^{[1]}} -
		\frac{\partial}{\partial\ell_{k-1}^{[1]}} \\
	  &= \frac{\partial}{\partial\ell_{m_3}^{[3]}} 
		- \frac{\partial}{\partial\ell_{1}^{[3]}} \\
  \tau_{\gamma_1 i} &=
	 \frac{\partial}{\partial \ell_{k+1}^{[1]}} 
	  - \frac{\partial}{\partial\ell_{m_1}^{[1]}} \\
	&= \frac{\partial}{\partial\ell_{m_2}^{[2]}} 
	  - \frac{\partial}{\partial\ell_{1}^{[2]}}, 
\end{align*}
from which we calculate
\begin{align*}
\iota_{\tau_i}\omega_1 &= -d\ell_{1}^{[1]} + d\ell_{3}^{[1]} + \cdots + d\ell_{m_1}^{[1]}
 + d\ell_{1}^{[1]} + \cdots + d\ell_{k-2}^{[1]} \\
 &\quad - (d\ell_{k}^{[1]} + \cdots + d\ell_{m_1}^{[1]}) \\
 	&= 2(d\ell_{2}^{[1]} + \cdots + d\ell_{k-1}^{[1]}) 
	- d\ell_{2}^{[1]} - d\ell_{k-1}^{[1]} \\
\iota_{\tau_i}\omega_3 &= -(d\ell_{2}^{[3]} + \cdots + d\ell_{m_3}^{[3]})
	- (d\ell_{1}^{[3]} + \cdots + d\ell_{m_3 - 1}^{[3]}) \\
	&= -2dp_3 + d\ell_{1}^{[3]} + d\ell_{m_3}^{[3]} \\
\iota_{\tau_{\gamma_1 i}}\omega_1 &= 2(d\ell_{k+1}^{[1]} + \cdots + d\ell_{m_1}^{[1]}) 
	- d\ell_{k+1}^{[1]} - d\ell_{m_1}^{[1]} \\
\iota_{\tau_{\gamma_1 i}}\omega_2 &=
-2dp_2 + d\ell_{1}^{[2]} + d\ell_{m_2}^{[2]}  .
\end{align*}
By canceling same-edge terms we have
\begin{align*}
\iota_{\tau_i}\Omega &= d(\ell_{2}^{[1]} + \cdots + \ell_{k-1}^{[1]}  ) - dp_3 \\
\iota_{\tau_{\gamma_1 i}} \Omega &= d(\ell_{k+1}^{[1]} + \cdots + \ell_{m_1}^{[1]})
	- dp_2.
\end{align*}
Ignoring the inconsequential $dp_k$ terms, we observe that
$\ell_{2}^{[1]} + \cdots + \ell_{k-1}^{[1]}$ is the period of the first
circle action, while $\ell_{k+1}^{[1]} + \cdots + \ell_{m_1}^{[1]}$ is
the period of the second, thus completing the proof of the lemma.
\end{proof}

A corollary of the above proof is that $\Omega$ restricted to
$\RG^{3}_{g,n}(L)$ is non-degenerate. In fact, the vector fields $T_i$
span the tangent space $T\RG^{3}_{g,n}(L)$, and the duality relation
$\iota_{T_i}\Omega = -2d\ell_i$ completely characterizes $\Omega$.

Note that an alternative description of the moment map is the perimeter
map for the newly created boundary (or boundaries) obtained by removing
edge $[i]_1$. Hence, the symplectic quotients are identified with
subsets of
ribbon graph complexes obtained by edge removal. To be precise, the
symplectic quotient is a subset of $\RG_{g',n'}$, where $(g', n')$ is
the type of the graph $\Gamma_{\hat{i}}$. In case removing $i$
disconnects $\Gamma$ into two graphs of type $(g_1,n_1)$ and
$(g_2, n_2)$, then the quotient will be a subset of $\RG_{g_1, n_1}
\times \RG_{g_2, n_2}$. Moreover, the perimeters of the newly created
graphs are fixed by the original perimeters of $\Gamma$ and the
particular level of the moment map taken for the quotient. 
A more precise determination of these perimeters and the types of
graphs which appear for the quotient is deferred to
Section~\ref{sect:RecursionFormula}. 

A consequence of the geometry of the quotients is that they have two
independent symplectic structures: $\overline{\Omega}$ coming from
symplectic reduction, and $\Omega$ induced from the Kontsevich
symplectic form on $\RG_{g',n'}$. Although they are defined
differently, the two symplectic structures agree:

\begin{lemma} $\bar{\Omega} = \Omega$. \end{lemma}

\begin{proof} 
  Recall that $\Omega$ is identified by the duality relation
  $\Omega(T_i, \cdot) = -2d\ell_i$, while $\bar{\Omega}$ is calculated
  by lifting vectors to the torus orbit. We denote the quotient map by
  $q: \TorusCover(\Gamma, i) \rightarrow \RG_{g',n'}(L')$, where the
  exact type and boundaries of the ribbon graph complex in the image
  is one of the possibilities discussed above. The torus quotient is a
  local operation, and any edge $j$ not incident to $i$ has $q_*T_j =
  T_j$, so it remains to find lifts of edges labeled $k_1$ and $k_2$.
  in Figure~\ref{fig:QuotientForm}
  \begin{figure}
	\begin{center}
	  \begin{tikzpicture}
		\draw (0,0) --node[anchor=north] {$\ell_{i_4}$} ++(.8,-0.3) -- +(0.25,0.25)
	+(0,0.) -- +(0.25, -0.25)
	(0,0) -- node[anchor=north] {$\ell_{i_3}$} ++(-0.8, -0.3) -- +(-0.25, 0.25)
  (-0.8,-0.3) -- +(-0.25, -0.25)
  (0,1) -- node[anchor=south] {$\ell_{i_1}$} ++(.8, 0.3) -- + (0.25, 0.25)
  (0.8, 1.3) -- +(0.25, -0.25)
  (0,1) -- node[anchor=south] {$\ell_{i_2}$} ++(-0.8, 0.3) -- + (-0.25, 0.25)
  (-0.8, 1.3) -- +(-0.25, -0.25)
  [->] (1.25, 0.5) -- node[anchor=north] {$q$} (2.25, 0.5);
  \fill  (0,0) circle (1pt) (0,1) circle (1pt);
  \draw[very thick] (0,0) -- node[anchor=west] {$\ell_i$} +(0,1);
  \draw [xshift=3.5cm] (-.8, -0.3) .. node[anchor=north] {$\ell_{k_2}$}
  controls (0,0) .. (.8,-0.3)
   -- +(0.25,0.25)
	+(0,0.) -- +(0.25, -0.25)
	(-0.8, -0.3) -- +(-0.25, 0.25)
  (-0.8,-0.3) -- +(-0.25, -0.25)
  (-.8, 1.3) .. node[anchor=south] {$\ell_{k_1}$} controls (0,1) .. (.8, 1.3) -- + (0.25, 0.25)
  (0.8, 1.3) -- +(0.25, -0.25)
  (-0.8, 1.3) -- + (-0.25, 0.25)
  (-0.8, 1.3) -- +(-0.25, -0.25);
	  \end{tikzpicture}
	\end{center}
	\caption{Edge notations used to calculate the quotient symplectic
	form.}
	\label{fig:QuotientForm}
  \end{figure}
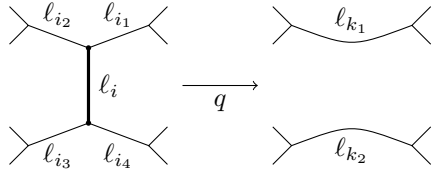
  
  However, it is clear that 
  \begin{align*}
	q_*(T_{i_1} + T_{i_2}) &= T_{k_1} \\
	q_*(T_{i_3} + T_{i_4}) &= T_{k_2},
  \end{align*}
  while
  \begin{align*}
	\Omega(T_{i_1} + T_{i_2}, \cdot) &= -2(d\ell_{i_1} +
	d\ell_{i_2}) \\
	&= -2d\ell_{k_1} \\
	\Omega(T_{i_3} + T_{i_4}, \cdot) &= -2(d\ell_{i_3} +
	d\ell_{i_4}) \\
	&= -2d\ell_{k_2}.
  \end{align*}
  This completes the proof of the lemma.
\end{proof}


\section{Recursion formula}
\label{sect:RecursionFormula}

As constructed in the previous section, we have a partition of unity subordinate to the open cover 
$$\{ \TorusCover(\Gamma, i) \, | \, \Gamma \in  \RG^3_{g,n}(L),\ b(i) = 1\}.$$
Hence we wish to calculate the partition-scaled volume of each $\TorusCover(\Gamma, i)$. Rather than calculate each individually, we will group the covers together according to the type and boundary labelings of the edge-deleted graph $\Gamma_{\hat{i}}$. In particular, the different types that arise are:

\begin{enumerate} 
	
  \item
Edge $i$ bounds perimeters $1$ and $j$ for some $j\neq 1$. In this case,
removing edge $i$ is the same as removing a $\theta$-graph with boundary
lengths $(L_1, L_j, x)$, leaving $\Gamma_{\hat{i}} \in \RG_{g, n-1}(x,
L_{\nset{n} \setminus \{ 1, j \}})$ where $\abs{L_1 - L_j} < x < L_1 +
L_j$. The length of the edge being removed (value of the partition of
unity) is calculated to be
\begin{equation*}
\ell_i = \frac{1}{2}(L_1 + L_j - x  )
\end{equation*}

\item
  Edge $i$ is part of a lollipop, with boundaries $1$ and $j$ on either
side (again, $1 \neq j$). We have $\Gamma_{\hat{i}}\in \RG_{g, n-1}(x,
L_{\nset{n} \setminus \{ 1, j \}})$ where $0 \leq x \leq \abs{L_1
- L_j}$. The total length of the lollipop (sum of all edges which
contribute to this term) is
\begin{equation*}
\begin{cases}
L_1 & \text{if $L_1 < L_j$} \\
L_1  - x & \text{if $L_1 > L_j$}
\end{cases}
\end{equation*}

\item Edge $i$ has boundary 1 on both sides, neither vertex has a
loop, and removing edge $i$ does not disconnect the graph. Then
$\Gamma_{\hat{i}}\in \RG_{g-1, n+1}(x, y, L_{\nset{n} \setminus \{
1,\}})$, where $0 < x+y < L_1$. The length of edge $i$ is calculated to
be $$ \ell_i = \frac{1}{2}(L_1 - x - y). $$

\item
  Edge $i$ has boundary 1 on both sides, neither vertex
is incident to a loop, and removing edge $i$ disconnects the graph. Then
$\Gamma_{\hat{i}}\in \RG_{g_1, n_1}(x, L_{\mathcal{I}}) \times
\RG_{g_2, n_2}(y, L_{\mathcal{J}})$, where $g_1 + g_2 = g$,
$\mathcal{I} \sqcup \mathcal{J} = \nset{n} \setminus \{ 1\}$, $n_1 =
\FiniteCount{\mathcal{I}} + 1$ and $n_2 = \FiniteCount{\mathcal{J}} +1$.
The newly created boundaries satisfy $0 < x+y < L_1$ and the length of
the removed edge is 
\begin{equation*} 
  \ell_i = \frac{1}{2}(L_1 - x- y). 
\end{equation*}
We note for future reference that if $g_1 = g_2$ and
$\FiniteCount{\mathcal{I}} = 0 = \FiniteCount{\mathcal{J}}$ then there
is a symmetry of order 2 obtained by exchanging $x$ and $y$.
\end{enumerate}

To be clear, there are multiple groupings coming from each of the types
listed above. For example, pairs $(\Gamma, i)$ satisfying type (1) with
$j=2$ are in a different group than pairs satisfying type (1) with
$j=3$. Only type (3) describes a single group.

What makes the integration scheme work is the fact that the symplectic
quotients $\TorusCover(\Gamma, i) \symplecticQuotient \Torus$ of a fixed
type form a disjoint cover of the appropriate ribbon graph complex.
In other words, the combined reduced volumes of a given type coincides
with the volume of the ribbon graph complex of the specified type. This
is most easily seen by working backwards. For example, starting with a
graph $\Gamma \in \RG_{g, n-1}(x, L_{\nset{n} \setminus \{1, j \}})$,
and a point on the boundary of length $x$, there is a unique way to
recover a graph in $\RG_{g,n}(L_{\nset{n}})$: if $x < \abs{L_1 - L_j}$
then one must attach a lollipop to the marked point. If $x > \abs{L_1
- L_j}$ then one must attach a theta graph to the marked point (one of
the two vertices of the theta graph must be distinguished in order to
perform this operation unambiguously). The other cases are similar.

To calculate the volume of $\RG_{g,n}(L_{\nset{n}})$, we use
the partition of unity to write the volume as a sum over
torus covers $\TorusCover(\Gamma, i)$. We group the covers by
type and use the symplectic volume equation \eqref{eq:ReductionVolumeFormula}
 to calculate the contribution
from each grouping. The result is that
\begin{equation}
  \begin{split}
L_1 & \Vol_{g,n}(L_1,  \ldots, L_n) \\
 &  = \sum_{j=2}^n
	\int_{\abs{L_1-L_j}}^{L_1 + L_j}dx\, \frac{x}{2} (L_1 + L_j - x ) 
	\Vol_{g,n-1}(L_{\nset{n}\setminus\{1, j\}},x) \\
 & \quad  + \sum_{j=2}^n \int_{0}^{\abs{L_1-L_j}} dx\, x f(x, L_1, L_j)
	\Vol_{g,n-1}(L_{\nset{n}\setminus\{1, j\}},x) \\
 & \quad + \iint_{0 \leq x+y \leq L_1} dxdy\,\frac{xy}{2}(L_1 - x-y)
	\Vol_{g-1,n+1}(L_{\nset{n}\setminus 1},x,y) \\
 &  + \sum_{\substack{g_1 + g_2 = g \\ \mathcal{I} \disjoint \mathcal{J} = \nset{n}\setminus 1}}
	\iint_{0 \leq x+y \leq L_1} dxdy\,\frac{xy}{2}(L_1 - x-y)
	\Vol_{g_1,n_1}(L_{\mathcal{I}},x)\Vol_{g_2,n_2}(L_{\mathcal{J}},y),
  \end{split}
	\label{eq:RecursionFormula}
\end{equation}
where
\begin{equation*}
f(x, y, z) =
\begin{cases}
 y - x & \text{if $y > z$} \\
 y & \text{if $y < z$}.
\end{cases}
\end{equation*}

Note that in the formula above, each integral summand is coming from one
of the groupings enumerated above. The integrand consists of the product
of periods of the torus action ($x$ or $xy$) times the appropriate
weighting from the partition of unity (length of the edge being removed)
times the volume of the symplectic quotient.

The term coming from case (3) has double the edge weighting, due to the
fact that the edge appears twice when traversing the boundary. This
factor of 2 is compensated by a factor of $\frac{1}{2}$ coming from
the fact that when removing the edge from a graph, there is no way to
distinguish the two newly created boundaries. Thus, $\RG_{g-1, n+1}(x,
y, L_{\nset{n} \setminus 1})$ double counts because the $x$-length
boundary is distinguished from the $y$-length boundary.

The double-edge contribution in case (4) is compensated for because the
sum over $g_1 + g_2 = g$ and $\mathcal{I} \sqcup \mathcal{J} = \nset{n}
\setminus 1$ gives a double count over the groupings. The one exception
is when $g_1 = g_2$ and $n=1$. This case only appears once in the sum
(if at all), but the factor of $\frac{1}{2}$ is accounted for by the
order 2 symmetry of the underlying graph.

The above equation is a \emph{topological recursion} formula for the
volumes. In particular, the types of the ribbon graphs appearing on
the right hand side are simpler than the type on the left, where the
complexity of a graph of type $(g,n)$ can be measured by $2g - 2 +
n$. Implicit in the above computation is the fact that $(g,n)\neq
(0,3), (1,1)$. These can be considered the base cases for the
recursion, as all other volume computations can be reduced to knowing
$\Vol_{0,3}(L_{\nset{3}})$ and $\Vol_{1,1}(L)$. Luckily, these complexes
are simple enough to calculate their volumes by hand, which we now
proceed to do.

\subsection*{Volume of $\RG_{0,3}(L_1, L_2, L_3)$}

The ribbon graph complex of type $(0,3)$ has dimension $6g-6 + 2n =
0$, so we are  integrating $\Omega^0 = 1$ over a 0-dimensional space.
In other words, $\Vol_{0,3}(L_{\nset{3}})$ is a discrete count of
metric ribbon graphs of specified perimeter lengths. The set of all
ribbon graphs of type $(0,3)$ can be found in
Figure~\ref{fig:0-3Graphs}. Note that the automorphism groups are all
trivial. Furthermore, once the perimeters are fixed, there is a unique
metric ribbon graph which realizes those perimeters. 

In particular, if
$L_i + L_j > L_k$ for all distinct $i, j, k$ then only the theta graph
is possible. The set of perimeter equations
\begin{align*}
  \ell_1 + \ell_2 & = L_1 \\
  \ell_2 + \ell_3 &= L_2 \\
  \ell_1 + \ell_3 &= L_3
\end{align*}
has a unique solution with all $l_i > 0$.
If, for some $i, j, k$ we have  $L_i + L_j = L_k$ then the only
possible graph is the figure-eight, with the two boundary loops
labeled $i$ and $j$. Finally, if $L_i + L_j < L_k$ for some $i, j, k$
then the graph is a dumbell, with the loop boundaries labelled by $i$
and $j$.
We conclude that
\begin{equation*}
  \Vol_{0,3}(L_1, L_2, L_3) = 1.
\end{equation*}

Note that this is the only 0-dimensional ribbon graph complex, and
therefore it is the only case where a non-trivalent graph contributes
to the volume.

\subsection*{Volume of $\RG_{1,1}(L)$}

For graphs of type $(1,1)$, the dimension of the ribbon graph complex
is $6g-6+2n = 2$. Hence we integrate $\Omega$ over $\RG_{1,1}(L)$. As
illustrated in Figure~\ref{fig:1-1Graphs}, there is a single graph
with the correct number of edges, whose automorphism group is
$\ZZ_6$. If the edges are labeled in such a way that when traversing
the boundary we encounter (in order) $\ell_1, \ell_2, \ell_3, \ell_1,
\ell_2, \ell_3$, then we have
\begin{equation*}
  \Omega = d\ell_1 \wedge (d\ell_2 + d\ell_3) + d\ell_2\wedge d\ell_3.
\end{equation*}

If we fix $2(\ell_1 + \ell_2 + \ell_3) = L$ then $d\ell_3 = -d\ell_1 -
d\ell_2$ so that
\begin{equation*}
  \Omega \biggr\rvert_{\RG_{1,1}(L)} = d\ell_1 \wedge d\ell_2,
\end{equation*}
and we have
\begin{equation*}
  \iint\limits_{\ell_1 + \ell_2 \leq \tfrac{1}{2} L} \Omega = \frac{1}{8}
  L^2.
\end{equation*}
After dividing by the order of the automorphism group we conclude that
\begin{equation*}
  \Vol_{1,1}(L) = \frac{1}{48} L^2.
\end{equation*}

\section{Virasoro constraints}
\label{sect:Virasoro}
The DVV formula \cite{MR1083914}, or Virasoro constraints, for $\psi$-class
intersections is
\begin{equation}
  \begin{split}
	\intersect{\tau_{d_1} \cdots \tau_{d_n}}_{g} &=
	\sum_{j=2}^{n} \frac{(2d_1 + 2d_j - 1)!!}{(2d_1 + 1)!! (2d_j -
	1)!!} \intersect{\tau_{d_1 + d_j - 1} \tau_{d_{\nset{n} \setminus
	\{1, j\} }}}_{g} \\
	& \quad + \frac{1}{2} \sum_{a+b = d_1 - 2}
	\frac{(2a+1)!!(2b+1)!!}{(2d_1+1)!!}
	\Biggl[ \intersect{\tau_a\tau_b \tau_{d_{\nset{n}\setminus
	1}}}_{g-1}  \\
	&  \quad\quad\quad + \sum_{\substack{g_1 + g_2 = g \\ \mathcal{I} \sqcup
	\mathcal{J} = \nset{n} \setminus 1}}^{\text{stable}} 
	\intersect{\tau_a\tau_{d_{\mathcal{I}}}}_{g_1}
	\intersect{\tau_b \tau_{d_{\mathcal{J}}}}_{g_2}\Biggr],
  \end{split}
  \label{eq:DVV}
\end{equation}
where we are using the notation
\begin{equation*}
  \intersect{\tau_{d_1} \cdots \tau_{d_n}}_g =
  \int_{\CompactModuli_{g,n}}\psi_1^{d_1}\cdots \psi_n^{d_n},
\end{equation*}
and
\begin{equation*}
  (2k+1)!! = (2k+1)(2k-1) \cdots 1 = \frac{(2k+1)!}{2^k k!}.
\end{equation*}
The stable sum in the last term means we restrict to terms where $(g_i,
n_i)$ satisfy $2g_i - 2 + n_i > 0$, where $n_1 =
\FiniteCount{\mathcal{I}} + 1$ and $n_2 = \FiniteCount{\mathcal{J}} +
1$. 

The topological recursion formula \eqref{eq:RecursionFormula} is
equivalent to the above DVV relation, if one looks at terms of fixed
degree in the $L_i$'s. To that end, if $P(L_{\nset{n}})$ is a
polynomial in $L_1^2, \ldots, L_n^2$ then we denote
\begin{equation*}
  [d_1 \cdots d_n] P(L_{\nset{n}})
\end{equation*}
to be the coefficient in $P$ of the monomial $L_1^{2d_1}\cdots
L_n^{2d_n}$. As seen in \eqref{eq:VolumePsiRelation}, we have
\begin{align*}
  [d_1 \cdots d_n] \Vol_{g,n}(L_{\nset{n}}) &= \frac{1}{\prod
  2^{d_i} d_i!} \int_{\CompactModuli_{g,n}} \psi_1^{d_1} \cdots
  \psi_n^{d_n} \\
  &= \frac{1}{\prod 2^{d_i}d_i!} \intersect{\tau_{d_1}\cdots
  \tau_{d_n}}_g,
\end{align*}
where the coefficient is non-zero if and only if $d_1 + \cdots + d_n =
d = 3g-3 + n$.

Note, however, that the topological recursion formula has
$L_1\Vol_{g,n}$ on the left hand side. In fact, it turns out to
simplify the calculations if we consider the differentiated
topological recursion equation - namely differentiate both sides by
$L_1$. Then the left hand side gives
\begin{equation}
  [d_1 \cdots d_n] \frac{\partial}{\partial L_1} L_1
  \Vol_{g,n}(L_{\nset{n}}) = 
  \frac{2d_1 + 1}{\prod 2^{d_{i}} d_{i}!}
  \intersect{\tau_{d_{1}}\cdots \tau_{d_n}}_g.
  \label{eq:lhsRecursionCoefficients}
\end{equation}

The differentiated topological recursion formula becomes somewhat
simpler:
\begin{equation}
  \begin{split}
	\frac{\partial}{\partial L_1} L_1 &\Vol_{g,n}  (L_{\nset{n}}) = \\ 
	& \quad \sum_{j=2}^{n} \int_{0}^{L_1 + L_j}\frac{x}{2} \Vol_{g,
	n-1}(x, L_{\nset{n} \setminus \{1, j\}})dx \\
	&  + \sum_{j=2}^{n}
	\int_{0}^{\abs{L_1 - L_j}} \frac{x}{2} \Vol_{g, n-1}(x,
	L_{\nset{n} \setminus \{1, j\}}) dx \\
	&  + \iint_{0 \leq x+y \leq L_1} \frac{xy}{2} \Vol_{g-1,
	n+1}(x, y, L_{\nset{n} \setminus \{1\}})dx\,dy \\
	&  + \sum_{\substack{g_1 + g_2 = g \\ \mathcal{I} \sqcup
	\mathcal{J} = \nset{n} \setminus \{1 \}}}
	\iint_{0 \leq x+y \leq L_1} \frac{xy}{2} \Vol_{g_1, n_1}(x,
	L_{\mathcal{I}}) \Vol_{g_2, n_2}(y, L_{\mathcal{J}}) dx\,dy.
  \end{split}
  \label{eq:DiffRecursionFormula}
\end{equation}

To calculate the matching monomial coefficient on the right hand side of
the recursion formula, we must explicitly evaluate the above integrals.

For a fixed integer $k \geq 0$ we have
\begin{equation}
  \begin{split}
	\int_{0}^{L_1 + L_j} \frac{x}{2}x^{2k}\,dx +
	\int_{0}^{\abs{L_1 - L_j}} & \frac{x}{2}x^{2k}\,dx  \\
	& = \frac{1}{2(2k+2)}(L_1 + L_j)^{2k+2} +
	\frac{1}{2(2k+2)}(L_1 - L_j)^{2k+2} \\
	&= \frac{1}{2(2k+2)} \sum_{r=0}^{2k+2}
	  \binom{2k+2}{r}
	  \left[ L_1^r L_j^{2k+2-r} + (-1)^r L_1^r L_j^{2k+2 - r} \right]
	  \\
	  &=  \sum_{s=0}^{k+1} \frac{(2k+1)!}{(2s)!(2k+2 - 2s)!}
	  L_1^{2s} L_j^{2(k+1 - s)}.
  \end{split}
  \label{eq:IntegratedUnstableTRF}
\end{equation}
By comparing degrees, we see that the coefficient of the $[d_1\cdots
d_n]$ term coming from the above integral is equal to
\begin{multline*}
    \frac{\bigl(2(d_1 + d_j) - 1\bigr)!}{(2d_1)!(2d_j)!}
   \bigl[d_1+d_j-1\,
   \prod_{i\neq 1,j}d_i \bigr] \Vol_{g, n-1} (x, L_{\nset{n}
  \setminus \{ 1, j\}}) \\
    = \frac{\bigl(2(d_1 + d_j) - 1\bigr)!}{(2d_1)!(2d_j)!}
  \frac{\intersect{\tau_{d_1 + d_j
  - 1} \tau_{\nset{n} \setminus\{1,j\}}}_g}{2^{d_1 + d_j - 1}2^{d_{\nset{n} \setminus \{1,j\}}} (d_1 +
  d_j - 1)! d_{\nset{n}\setminus\{1,j\}}!} .
\end{multline*}

We also calculate the double integrals by fixing integers $a, b
\geq 0$:
\begin{equation*}
  \iint\limits_{0 \leq x+y \leq L_1} \frac{xy}{2} x^{2a}y^{2b}\,dx\,dy
  = \frac{1}{2} \frac{(2a+1)! (2b+1)!}{(2(a + b +
  2))!}L_1^{2(a + b + 2)}.
\end{equation*}
Hence terms containing $L_1^{2d_1}$ must have $a+b = d_1 - 2$. 

Assembling the individual contributions yields
\begin{equation*}
  \begin{split}
	\frac{2d_1 + 1}{\prod 2^{d_i} d_i !} &
	\intersect{\tau_{d_1} \cdots \tau_{d_n}}_g \\
  &= \sum_{j=2}^{n} 
   \frac{\bigl(2(d_1 + d_j) - 1\bigr)!}{(2d_1)!(2d_j)!}
  \frac{\intersect{\tau_{d_1 + d_j
  - 1} \prod_{i\neq 1, j}\tau_{d_i}}_g}{2^{d_1 + d_j - 1} (d_1 +
  d_j - 1)! \prod_{i\neq 1, j} 2^{d_i}d_i!} \\
  & \quad + 
  \frac{1}{2} \sum_{a+b = d_1 - 2} \frac{(2a+1)!(2b+1)!}{(2d_1)!}
  \frac{1}{2^a a! 2^b b!}\frac{1}{\prod_{i\neq 1}2^{d_i} d_i!}
  \Biggl[ \intersect{\tau_a\tau_b \prod_{i\neq
  1}\tau_{d_i}}_{g-1} \\
  & \quad\quad + 
  \sum_{\substack{g_1 + g_2 = g \\ \mathcal{I} \sqcup \mathcal{J} =
  \nset{n} \setminus\{1\}}} \intersect{\tau_a \prod_{i \in
  \mathcal{I}}\tau_{d_i}}_{g_1} 
  \intersect{\tau_b \prod_{i\in\mathcal{J}}\tau_{d_i}}_{g_2}
  \Biggr].
  \end{split}
\end{equation*}

Canceling matching terms and using the relation
\begin{equation*}
  (2k-1)!! = \frac{(2k)!}{2^k k!}
\end{equation*}
gives the DVV equation \eqref{eq:DVV}.

We note that \cite{MR2257394, MR2483941, mulase-2009} all obtained the same
results, in some cases using similar
ideas. However, in those situations a scaling limit argument was
always needed to access the $\psi$-class terms. In the present case,
the derivation of the DVV is much simpler, due to the fact that no
rescaling is necessary.
\section{Eynard-Orantin recursion}
\label{sect:EO_Recursion}
In this section we prove that the topological recursion formula
presented above is equivalent to the Eynard-Orantin recursion for the
spectral curve $x = \frac{1}{2}y^2$. The main idea is to take the
Laplace transform the the recursion formula.

To that end, we define
\begin{equation*}
  W_{g,n}( z_1, \ldots,  z_n) = \idotsint e^{-z_{\nset{n}}\cdot
  L_{\nset{n}}}
  L_1 \cdots L_n \Vol_{g,n}(L_1, \ldots, L_n)
  dL_1\cdots dL_n,
\end{equation*}
where $z_{\nset{n}}\cdot L_{\nset{n}} =  z_1L_1 + \cdots z_n L_n  $ and the integration is
performed over $[0, \infty)^n$. If
we take the Laplace transform of $L_2\cdots L_n$ times the recursion formula then the left
hand side becomes $W_{g,n}( z_1, \ldots,  z_n)$. To evaluate
the integrals on the right hand side we swap the order of integration,
as explained below.

In general, if $V(x,y)$ denotes a polynomial in $x^2$ and $y^2$ and
\begin{equation*}
  W(z_1,z_2) = \int_{0}^{\infty}\int_{0}^{\infty} e^{-xz_1 -
  yz_2}xyV(x,y) dxdy
\end{equation*}
then
\begin{equation}
  \begin{split}
  \int_{0}^{\infty}dL_1 & e^{- z_1 L_1} \int\int_{0\leq x+y \leq
  L_1} dxdy xy(L_1 - x- y)V(x,y) \\
  &= \int_{0}^{\infty}dx \int_{0}^{\infty}dy \int_{x+y}^{\infty}dL_1 xy
  \left( -\frac{\partial}{\partial z_1} - (x+y)
  \right)e^{- z_1L_1}V(x,y) \\
  &= \frac{1}{ z_1^2}\int_{0}^{\infty}dx\int_{0}^{\infty}dy xy
  V(x,y)e^{- z_1x -  z_1 y} \\
  &= \frac{1}{ z_1^2}W( z_1,  z_1).
  \end{split}
  \label{eq:StableLaplace}
\end{equation}

For the term involving boundary $j$, we define
\begin{equation*}
	F(x, L_1, L_j) = 
  \begin{cases}
	L_1 & \text{if $L_1 < L_j$, $x < L_j-L_1$} \\
	L_1 - x & \text{if $L_j < L_1$, $x < L_1 - L_j$} \\
	\frac{1}{2}(L_1 + L_j - x) & \text{if $\abs{L_1 - L_j} < x < L_1 +
	L_j$},
  \end{cases}
\end{equation*} then calculate
\begin{multline}
  \int_{0}^{\infty}dL_j L_j e^{- z_jL_j}\int_{0}^{\infty}dL_1
  e^{- z_1L_1}\int_{0}^{L_1 + L_j} dx x F(x, L_1, L_j) V(x) \\
  =-\frac{\partial}{\partial z_j}
  \left[ \frac{1}{ z_1^2 z_j( z_1^2 -  z_j^2)}
  \left(  z_1^2 W( z_j) -  z_j^2W( z_1)
  \right)\right],
  \label{eq:UnstableLaplace}
\end{multline}
where we have again adopted the convention that $W(z)$ is the Laplace
transform of $xV(x)$, under the assumption that $V(x)$ is a polynomial
in $x^2$.
This equation can be verified by Mathematica, or calculated easily by hand by
swapping the order of integration and
introducing the change of variables
\begin{align*}
  u &= L_1 + L_j - x \\
  v &= L_j - L_1.
\end{align*}

Coming back to the topological recursion formula
\eqref{eq:RecursionFormula}, we use \eqref{eq:UnstableLaplace} to
calculate the Laplace transform of the first line of the formula, and
\eqref{eq:StableLaplace} for the remainder of the terms. The result is
a recursion formula for $W_{g,n}$:
\begin{lemma}
\begin{equation}
  \begin{split}
  W_{g,n}(z_{\nset{n}}) &= \sum_{j=2}^{n} -\frac{\partial}{\partial z_j}
  \left[ \frac{z_j}{(z_1 z_j)^2(z_1^2 - z_j^2)} 
  \left( z_1^2 W_{g, n-1}(z_{\nset{n} \setminus 1}) 
  - z_j^2 W_{g, n-1}(z_{\nset{n}\setminus j}) \right)\right] \\
  & \quad+ \frac{1}{2z_1^2} W_{g-1, n+1}(z_1, z_{\nset{n}}) \\
  & \quad+ \frac{1}{2z_1^2} \sum_{\substack{g_1 + g_2 = g \\ \mathcal{I}
  \sqcup \mathcal{J} = \nset{n} \setminus 1}} 
  W_{g_1, n_1}(z_1, z_{\mathcal{I}})W_{g_2, n_2}(z_1,
  z_{\mathcal{J}}),
\end{split}
  \label{eq:LaplaceRecursionFormula}
\end{equation}
where $n_1 = \FiniteCount{\mathcal{I}} + 1$, $n_2 =
\FiniteCount{\mathcal{J}} + 1$ and the summation in the last line is
taken over all pairs $(g_1, \mathcal{I})$, $(g_2, \mathcal{J})$
subject to the stability condition $2g_i - 2 + n_i > 0$.
\end{lemma}

The goal is to equate this recursion formula to Eynard-Orantin
recursion for the Airy curve $x = \frac{1}{2}y^2$. To do so, we need
to explicity evaluate the residues involved in
\eqref{eq:AiryResidueRecursion}. As a
starting point, if we have a function $W(\zeta_1, \zeta_2)$, which we
assume to be a polynomial in $\zeta_1^{-2}$ and $\zeta_2^{-2}$ then we
have
\begin{align*}
  \Res_{\zeta \rightarrow 0} E(\zeta, z_1) W(\zeta, -\zeta) d\zeta
  \otimes d(-\zeta) 
  &= \Res_{\zeta \rightarrow 0} \frac{-1}{2\zeta} \frac{1}{\zeta^2 -
  z_1^2} W(\zeta, \zeta)d\zeta \otimes dz_1 \\
  &= \frac{1}{2z_1^2} W(z_1, z_1) dz_1.
\end{align*}
Recall that the Eynard kernel $E(\zeta, z_1)$ is equal to
\begin{equation*}
  E(\zeta, z_1) = \frac{1}{2\zeta(\zeta^2 - z_1^2)}
  \frac{1}{d\zeta}\otimes dz_1.
\end{equation*}

The unstable terms in the Eynard-Orantin recursion have a more
complicated residue calculation, due to the diagonal pole in
$W_{0,2}$. We have
\begin{multline*}
  \Res_{\zeta\rightarrow 0} E(\zeta, z_1) 
  \bigl( W_{0,2}(\zeta, z_j)W(-\zeta) + W_{0,2}(-\zeta, z_j)W(\zeta)
  \bigr)
  d\zeta \otimes d(-\zeta) \\
  = \Res_{\zeta\rightarrow 0} \frac{1}{2\zeta} \frac{-1}{\zeta^2 -
  z_1^2}
  \left[ \frac{1}{(\zeta - z_j)^2} + \frac{1}{(\zeta + z_j)^2}
  \right]W(\zeta) d\zeta \otimes dz_1 \\
  = \Res_{\zeta\rightarrow 0} \frac{1}{\zeta( z_1^2 - \zeta^2 )} 
  \left[ -\frac{\partial}{\partial z_j} \frac{z_j}{z_j^2 - \zeta^2}
  \right]W(\zeta) d\zeta \otimes dz_1,
\end{multline*}
where we are assuming that $W(\zeta)$ is a polynomial in
$\zeta^{-2}$. To finish the calculation, we have the following
\begin{lemma}
For any integer $k \geq 0$
\begin{equation*}
  \Res_{\zeta\rightarrow 0} \frac{1}{\zeta}\frac{1}{z_1^2 -
  \zeta^2}\frac{1}{z_j^2 - \zeta^2} \zeta^{-2k} d\zeta
  =
  \frac{1}{z_1^2 z_j^2 (z_1^2 - z_j^2)} \left( z_1^2 z_j^{-2k} -
  z_j^2z_1^{-2k} \right).
\end{equation*}
\end{lemma}
\begin{proof}
 We expand the left hand side, assuming that $\zeta$ is closer to 0
 than both $z_1$ and $z_j$:
 \begin{align*}
   \Res_{\zeta\rightarrow 0} \frac{1}{z_1^2 -
   \zeta^2}\frac{1}{z_j^2 - \zeta^2} \zeta^{-2k-1}d\zeta
   &= \Res_{\zeta\rightarrow 0} \frac{1}{z_1^2 z_j^2} 
   \left( 1 + \frac{\zeta^2}{z_1^2} + \cdots  \right)
   \left( 1 + \frac{\zeta^2}{z_j^2} + \cdots \right)
   \zeta^{-2k-1}d\zeta \\
   &= \frac{1}{z_1^2 z_j^2} \sum_{r+s = k} z_1^{-2r}z_j^{-2s} \\
   &= \frac{z_1^2 - z_j^2}{z_1^2 z_j^2(z_1^2 - z_j^2)} \sum_{r=0}^{k}
   z_1^{-2r} z_j^{-2(k-r)} \\
   &= \frac{1}{z_1^2 z_j^2(z_1^2 -
   z_j^2)} ({z_1^2 z_j^{-2k} - z_1^{-2k}z_j^2}).
 \end{align*}
\end{proof}

Putting everything together gives us
\begin{theorem}
  The topological recursion formula \eqref{eq:RecursionFormula} for the symplectic
  volume of the ribbon graph complex is equivalent to the
  Eynard-Orantin recursion for the spectral curve $x =
  \frac{1}{2}y^2$.
\end{theorem}

\bibliographystyle{hplain}
\bibliography{References}

\end{document}